\documentclass[11pt]{article}

\usepackage[T1]{fontenc}
\usepackage{lmodern}
\usepackage[utf8]{inputenc}
\usepackage[english]{babel}

\usepackage[all]{xy}
\usepackage{amsmath, amscd}
\usepackage{amssymb}
\usepackage{amsthm}

\usepackage{graphicx}
\usepackage{psfrag}

\usepackage[small]{caption}

\usepackage{enumerate}

\theoremstyle{plain}
\newtheorem{thm}{Theorem}[section]
\newtheorem*{thm*}{Theorem}
\newtheorem{prop}[thm]{Proposition}
\newtheorem*{prop*}{Proposition}
\newtheorem*{assumption*}{Assumption}

\theoremstyle{definition}
\newtheorem{defin}[thm]{Definition}
\newtheorem{cor}[thm]{Corollary}
\newtheorem*{cor*}{Corollary}

\newtheorem{ex}[thm]{Example}

\theoremstyle{remark}
\newtheorem{rem}[thm]{Remark}
\newtheorem*{rem*}{Remark}

\newtheorem*{notation*}{Notation}

\newtheoremstyle{citing}
  {\topsep}
  {\topsep}
  {\itshape}
  {}
  {\bfseries}
  {.}
  {.5em}
  {\thmnote{#3}}

\theoremstyle{citing}

\newcommand{\field}[1]{\mathbb{#1}}

\newcommand{\R}{\field{R}}

\newcommand{\N}{\field{N}}
\newcommand{\Z}{\field{Z}}

\newcommand{\ron}[1]{\mathcal{#1}}

\newcommand{\nr}[1]{\left\Vert #1\right\Vert}

\newcommand{\ie}{{\emph{i.e.},~}}

\newcommand{\eps}{\varepsilon}
\renewcommand{\phi}{\varphi}
\renewcommand{\bar}[1]{\overline{#1}}

\newcommand{\ra}{\rightarrow}

\newcommand{\lra}{\longrightarrow}

\usepackage{verbatim, amscd}

\newcommand{\qtran}{quasi-translation{}}
\newcommand{\qtrans}{quasi-translations{}}

\title{Tiling deformations, cohomology, and orbit equivalence of tiling spaces}

\date{}

\author{Antoine Julien \\ {\small Nord University,} \\ {\small Levanger, Norway} \\ {\small antoine.julien@nord.no}
\and
Lorenzo Sadun \\ {\small Department of Mathematics, University of Texas,} \\ {\small Austin, TX 78712, USA} \\ {\small sadun@math.utexas.edu}
}

\begin{document}

\maketitle

\begin{abstract}
  We study homeomorphisms of minimal and uniquely ergodic tiling
  spaces with finite local complexity (FLC), of which suspensions of
  (minimal and uniquely ergodic) $d$-dimensional subshifts are an
  example, and orbit equivalence of tiling spaces with (possibly)
  infinite local complexity (ILC). In the FLC case, we construct a
  cohomological invariant of homeomorphisms, and show that all
  homeomorphisms are a combination of tiling deformations, 
maps homotopic to the identity (known as \qtrans),
  and local equivalences (MLD). In the ILC case, we construct a
  cohomological invariant in the so-called weak cohomology, and show
  that all orbit equivalences are combinations of tiling deformations,
  \qtrans, and topological conjugacies. These generalize results of
  Parry and Sullivan to higher dimensions.  We also show that
  homeomorphisms (FLC) or orbit equivalences (ILC) are completely
  parametrized by the appropriate cohomological invariants.  Finally,
  we show that, under suitable cohomological conditions, continuous
  maps between tiling spaces are homotopic to compositions of tiling
  deformations and either local derivations (FLC) or factor maps
  (ILC).

 \smallskip
 \noindent
{\small {\it MSC codes:} 37A20 (primary) 37B50, 52C23 (secondary)}
\end{abstract}


\section{Introduction}

In this paper, we study equivalences
between minimal and uniquely ergodic
aperiodic tiling spaces. (By ``aperiodic'' we mean that
the tilings do not admit any translational symmetries.)
Such spaces in dimension $d$ are compact spaces carrying an
$\R^d$-action.
A simple example is given by the suspension of a minimal
$\Z^d$-subshift. 
We distinguish between tilings with 
\emph{finite local complexity}, or FLC,
\emph{i.e.}, which have finitely many local configurations,
and those with \emph{infinite local complexity}, or ILC. See 
\cite{FS12} for a number of constructions of ILC tiling spaces, 
and \cite{Bellissard-talk, Bel15} for discussions of the physical
significance of ILC tilings. 

When dealing with FLC tiling spaces, the natural notion of 
equivalence is \emph{homeomorphism}. When dealing with ILC tiling spaces,
we consider \emph{orbit-equivalences}, \emph{i.e.}, homeomorphisms that
send $\R^d$ orbits to $\R^d$ orbits. In the FLC case, the path components of
the tiling spaces are $\R^d$ orbits, which guarantees that all homeomorphisms
are orbit-equivalences. 

In both settings, we ask two questions:
\begin{enumerate}
\item If $\Omega$ and $\Omega'$ are orbit-equivalent tiling spaces, how
do the geometric and combinatorial structures of tilings in $\Omega$ relate
to the geometric and combinatorial structures of tilings in $\Omega'$? 
Put another way, how different can two tilings be and still belong to 
orbit-equivalent tiling spaces? 
\item If $h: \Omega \to \Omega'$ is an orbit-equivalence, can we decompose
$h$ as the composition of maps, each of which preserves particular features
of the tilings?
\end{enumerate}

The answer to the second question is ``yes'', and this decomposition 
gives a partial answer to the first question. To state the result 
for the FLC case, we need some terminology. 

We call 
a map $\tau_s: \Omega \to \Omega$ that is homotopic to the 
identity 
a \emph{\qtran}. 
The term refers to the fact that $\tau_s(T)$ is a translate of 
$T$, and the prefix ``quasi'' acknowledges that this translation 
depends on $T$. Equivalently, there is a continuous function $s: \Omega \to
\R^d$, and $\tau_s(T)=T-s(T)$. (See Proposition \ref{prop6-2} below 
and \cite{Kwa10}
for proofs that $s$ is continuous.)

In a \emph{shape change transformation} $h_{sc}: \Omega \to \Omega''$,
the tilings in $\Omega''$ have exactly the same combinatorics as
the corresponding tilings in $\Omega$, only with tiles of different shape and
size. Shape changes are defined more precisely in Section 8. 
For a further discussion of shape changes, and their
classification using cohomology, see \cite{CS06, Kel08}.

\emph{Local derivations} are the tiling analogues of sliding block codes.
That is, $f: \Omega_1 \to \Omega_2$ is a local derivation if there exists a
radius $R$ such that, whenever $T\in \Omega_1$ and $T' \in \Omega_1$ agree
on a ball of radius $R$ around an arbitrary point $x \in \R^d$, then 
$f(T)$ and $f(T')$   
agree exactly on a ball of radius 1 around $x$. A \emph{mutual local
derivation} or \emph{MLD equivalence} 
is an invertible local derivation whose inverse is also 
a local derivation. All MLD equivalences are topological conjugacies, but
in contrast to symbolic dynamics, not all topological conjugacies 
are MLD equivalences \cite{RS01, Pet99}.

Our main result for tilings with FLC is then:

\begin{thm}\label{main-intro}
Suppose that 
$\Omega$ is a minimal and uniquely ergodic space of FLC aperiodic tilings, that
$\Omega'$ is a space of FLC tilings, and that $h:\Omega \to \Omega'$ is 
a homeomorphism. Then there exists an FLC tiling space $\Omega''$ and 
homeomorphisms 
$\tau_s:\Omega \to \Omega$, $h_{sc}: \Omega \to \Omega''$ and $\phi:
\Omega' \to \Omega''$, such that
\begin{itemize}
\item $\phi \circ h = h_{sc} \circ \tau_s$, or equivalently 
$h = \phi^{-1} \circ h_{sc} \circ \tau_s$,
\item $\tau_s$ is a \qtran{},
\item $h_{sc}$ is a shape change transformation, and 
\item $\phi$ is an MLD equivalence. 
\end{itemize}
\end{thm}

An almost identical theorem holds for tilings that may have ILC:
\begin{thm}\label{main-intro2}
Suppose that 
$\Omega$ is a minimal and uniquely ergodic space of aperiodic tilings, that
$\Omega'$ is a space of tilings, and that $h:\Omega \to \Omega'$ is 
an orbit equivalence.
Then there exists a tiling space $\Omega''$ and 
homeomorphisms 
$\tau_s:\Omega \to \Omega$, $h_{sc}: \Omega \to \Omega''$ and $\phi:
\Omega' \to \Omega''$, such that
\begin{itemize}
\item $\phi \circ h = h_{sc} \circ \tau_s$, or equivalently 
$h = \phi^{-1} \circ h_{sc} \circ \tau_s$,
\item $\tau_s$ is a \qtran{},
\item $h_{sc}$ is a shape change transformation, and 
\item $\phi$ is a topological conjugacy.
\end{itemize}
\end{thm}

We return to our first question.  
If $h:\Omega \to \Omega'$ is a homeomorphism of minimal and uniquely 
ergodic FLC tiling spaces, and if 
$T\in \Omega$ and $T'=h(T) \in \Omega'$ are corresponding tilings, 
then $T'$ is obtained from $T$ by 
\begin{enumerate}
\item Applying a translation to $T$ to get $\tau_s(T)$, 
\item Changing the shapes and sizes of the tiles of $\tau_s(T)$, while
preserving their combinatorics, to get a tiling $h_{sc}\circ \tau_s(T)$, 
and finally
\item Locally recoding, via $\phi^{-1}$, 
the patterns of $h_{sc}\circ\tau_s(T)$.
\end{enumerate}

In dimension one, these results were essentially known: a theorem of
Parry and Sullivan~\cite{PS75,PT82-book} states that
flow equivalence between two Markov shifts or minimal subshifts
(\emph{i.e.}, orbit-equivalence of their suspensions)
is generated by a re-labeling of the tiles and a change in the
length of the tiles (an MLD equivalence and a shape change).
If $X$ is a subshift with shift map $\sigma$, then the {\em suspension}
of $X$, denoted $SX$ and also called the {\em mapping cylinder of $\sigma$}, 
is the space $X\times \R / \sim$, where 
$(u,t) \sim (\sigma(u), t-1)$. If $r: X \to \R$ is a locally constant function
of constant sign, then $S^rX$, called the {\em suspension of $\sigma$
with roof function $r$}, is defined to be the quotient of $S \times \R$
by the relation $(u,t) \sim (\sigma(u), t-r(u))$. 

Parry and Sullivan proved that, if $X$ and $Y$ are two minimal subshifts, and
$\phi$ is a homeomorphism $SX \ra SY$,
then there exists a locally constant function $r: X \to \R$
of
constant sign $r$, and a topological conjugacy
\(
 \Phi: S^r X \lra SY
\).
In particular, $r$ defines an element in the first cohomology
group
\[
 H^1 (X, \Z) \otimes \R \simeq \Bigl( \{f: X \ra \Z \} / \langle f - f \circ \sigma \rangle \Bigr) \otimes \R,
\]
where $\sigma$ is the shift map on $X$.  The class $[r]$ is then either a
positive or negative element in the ordered cohomology group,
depending whether $r$ is positive or negative (and depending whether
$\phi$ preserves or changes the orientation).  Conversely, if $[f]$ is
any element in $H^1(X,\Z) \otimes \R$ whose pairing with every
invariant probability measure is (say) positive, then there exists a
representative $f_0 > 0$ of this cohomology class~\cite{BH96, GPS95}.
In this case, the suspension $S^{f_0}X$ is a well-defined object, and
in our formalism, the invariant for the obvious map $SX \ra S^{f_0}X$
is precisely $[f_0] = [f]$.  Hence, positive (resp.\@ negative)
elements of $H^1(X,\R)$ parametrize orientation preserving (resp.\@
reversing) flow equivalences from $X$ to other subshifts up to
conjugacy.\footnote{To be precise, $S^f X$ and $S^g X$ are conjugate
  whenever $(f-g)(x) = s(x) - s\circ\sigma(x)$ for a \emph{continuous}
  $s$, rather than locally constant~\cite{PT82,CS03}. The positive or
  negative elements of $H^1$ itself parametrize flow equivalence up to
  a stronger notion---called MLD below.}  An interesting point is that
in dimension~$1$, the subset of $H^1(X,\R)$ parametrizing flow
equivalences can be defined by means of an order structure, and does
not require any ergodicity conditions (compare with
Theorem~\ref{structure1} below).

\bigskip

We prove Theorem \ref{main-intro} in three
steps. First, we associate to every homeomorphism $\Omega \to \Omega'$ of 
minimal FLC tiling spaces a class 
$[h]$ in the first \v Cech cohomology $\check H^1(\Omega, \R^d)$ 
with coefficients
in $\R^d$. There are several isomorphic realizations of this 
cohomology, in particular the \emph{strong} pattern-equivariant cohomology
of \cite{Kel03}. (See Section 4 for a discussion of tiling cohomology.)
We prove
\begin{thm}[Theorem \ref{Classification-Thm}]\label{structure1}
If $\Omega$, $\Omega_1$, and $\Omega_2$ are minimal FLC tiling
spaces, if $h_1: \Omega \to \Omega_1$
and $h_2: \Omega \to \Omega_2$ are homeomorphisms, and if 
$[h_1]=[h_2]$, then there is a commutative diagram
\begin{equation} \label{structureeq1}
   \begin{CD}
    \Omega   @>{h_1}>> \Omega_1  \\
     @V{\tau_s}VV           @VV{\phi}V \\
    \Omega   @>{h_2}>> \Omega_2
   \end{CD}
\end{equation}
where $\tau_s$ is a \qtran{} and $\phi$
is an MLD equivalence. 
\end{thm}
This is a strengthening of a result \cite{Jul17} of 
the first author, who constructed a cohomological invariant that characterizes
homeomorphisms up to topological conjugacy and homotopy. This step does
not require unique ergodicity, only FLC and minimality. 

Next we consider what values the class $[h]$ may take. 
Here we restrict attention to uniquely ergodic systems.
Given the (unique) invariant measure $\mu$, the Ruelle--Sullivan
map $C_\mu$ \cite{KP06} 
sends elements of $\check H^1(\Omega; \R^d)$ 
to $d \times d$ square matrices. The construction also applies to
``weak'' cohomology.  In one dimension, the image
of a cocycle under the Ruelle--Sullivan map can be seen as an analogue
of Poincar\'e's rotation number: a $1$-cocycle valued in $\R$ can be
restricted to an orbit and integrated to a $0$-cochain\footnote{This
  is a $0$-cochain of $\R$ and not of $\Omega$, \emph{a
    priori}. Cochains of $\Omega$ are assumed to be
  \emph{pattern-equivariant}, see Section~\ref{sec:cohomology}.}%
, say $F$. Then by unique ergodicity,
\[
 \lim_{t \ra +\infty} \frac{F(x+t) - F(x)}{t}
\]
converges, and depends neither on $x$ nor on the orbit.

In higher dimensions, 
the matrix $C_\mu([h])$ has a similar description as 
the large-scale distortion
associated with the homeomorphism $h$. It should come as no surprise that
this distortion is never singular. 

\begin{thm}[Theorem \ref{RS-invertible}] \label{RS-theorem1}
Let $h: \Omega \ra \Omega'$ be a homeomorphism between two
minimal FLC tiling spaces, with $\Omega$ uniquely ergodic. 
Then $C_\mu ([h])$ is invertible. 
\end{thm}

Finally, we construct shape changes from 
cohomology classes in $\check H^1(\Omega; \R^d)$ whose images under
$C_\mu$ are invertible.

\begin{thm}[Theorem \ref{lastthm}]\label{RS-theorem2}
Let $\Omega$ be a minimal and uniquely ergodic space of FLC tilings,
let $\alpha \in \check H^1(\Omega, \R^d)$ and suppose that $C_\mu(\alpha)$ is
invertible. Then there exists a shape change transformation
$h_{sc}$ such that $[h_{sc}]=\alpha$.
\end{thm}

From these three results, Theorem \ref{main-intro} follows easily. If 
$h: \Omega \to \Omega'$ is a homeomorphism of minimal FLC tiling spaces
with $\Omega$ uniquely ergodic, then $C_\mu[h]$ is invertible, so there
exists a shape change $h_{sc}$ such that $[h_{sc}]=[h]$. 
The commutative diagram (\ref{structureeq1}), with $h_1=h$ and 
$h_2=h_{sc}$, 
is then tantamount to the conclusion of Theorem \ref{main-intro}.

There is an analogous development without the FLC assumption. Instead
of working in $\check H^1(\Omega, \R^d)$, we work with the so-called
\emph{weak cohomology} $H^1_w(\Omega, \R^d)$, described in Section 4,
and the Ruelle-Sullivan map is modified to map 
$H^1_w(\Omega, \R^d)$ to $d\times d$ square matrices. 
The statements of 
Theorems \ref{structure1}, \ref{RS-theorem1} and \ref{RS-theorem2} need
only small modifications, dropping the FLC requirement, replacing 
$\check H^1$ with $H^1_w$, replacing homeomorphisms with orbit-equivalences, 
and replacing MLD-equivalences with topological conjugacies. In fact, the 
statements and proofs of Theorems
\ref{Classification-Thm}, \ref{RS-invertible}, and \ref{lastthm} 
include both the FLC and non-FLC cases. By exactly the same argument as
with the FLC case, these three
theorems imply Theorem \ref{main-intro2}.

Note that, for uniquely ergodic FLC tiling spaces, the set
\[ S= \{ \alpha \in \check H^1(\Omega; \R^d) | C_\mu(\alpha) \in GL(d,\R)\}
\]
not only provides an invariant for tiling
spaces that are homeomorphic to $\Omega$, up to MLD. 
It actually provide a \emph{parametrization} of all such
homeomorphic spaces. This parametrization is unique, up to the action of 
self-homeomorphisms of $\Omega$ on $S$. Understanding
this group of self-homeomorphisms, and its action on cohomology, is a subject
for a future paper. 

Finally, we consider cohomological invariants for arbitrary maps between
minimal FLC tiling spaces. Given any continuous map $h: \Omega \to \Omega'$ 
of such spaces, we define a class 
$[h] \in \check H^1(\Omega;
\R^d)$. We then prove the following analogue of 
Theorem \ref{structure1}:

\begin{thm}[Theorem \ref{factor-thm}]\label{Any-map-theorem}
Let $h_1: \Omega \ra \Omega_1$ be a homeomorphism of minimal FLC tiling spaces 
and $h_2: \Omega \ra \Omega_2$ be a continuous map of minimal FLC tiling spaces.
If $[h_1] = [h_2]$ as elements of $\check H^1(\Omega, \R^d)$, then
$h_2$ is homotopic to a composition $\phi \circ h_1$, where $\phi: 
\Omega_1 \ra \Omega_2$ is a local derivation. 
\end{thm}
In particular, if $\Omega$ is uniquely ergodic and $C_\mu([h_2])$ is 
invertible, then we can take $h_1$ to be a shape change transformation. 
Every continuous and surjective map of an FLC, minimal and  
uniquely ergodic tiling space to another minimal FLC tiling space 
(possibly with an assumption about the invertibility of $C_\mu([h_2])$)
is then seen to be homotopic
to the composition of a shape change and a local derivation. 

As usual, there is an analogous theorem for ILC spaces, with 
$\check H^1$ replaced by $H^1_w$, with local derivations replaced by factor
maps, and with all maps required to send 
translational orbits to translational orbits. 

The structure of the paper is as follows. In Sections 2--3 we review
the formalism of tiling spaces and spaces of Delone sets, and in
Section 4 we review the different notions of tiling space cohomology
and the relations between them. Section 4 is particularly important, because
there are numerous versions of real-valued 
tiling cohomology, all known to be 
isomorphic to one of two groups: the \v Cech cohomology 
$\check H^k(\Omega, \R)$ or the weak cohomology $H^k_w(\Omega, \R)$. 
Various constructions in this paper use different realizations of these 
cohomology theories, and it is frequently necessary to shift from one
picture to another. For more information on tiling cohomology, see
\cite{BK10, Sad08, Sad15}.
In Section 5 we define the invariant
$[h] \in \check H^1(\Omega,\R^d)$ 
for homeomorphisms of FLC tiling spaces and compute it in a
number of examples. In Section 6 we define the invariant in
$H^1_w(\Omega; \R^d)$ for orbit equivalences of general tiling spaces,
and show how it relates to the previously defined invariant when the
tiling spaces have FLC. We also show how the invariants classify
homeomorphisms up to \qtran{} and either MLD equivalence or 
topological conjugacy, depending on the setting, thereby proving
Theorem \ref{structure1} and its ILC analogue.
In Section 7 we define the
Ruelle--Sullivan map and prove Theorem \ref{RS-theorem1} and its ILC
analogue. 
In Section 8 we prove Theorem \ref{RS-theorem2} and its ILC analogue, thereby 
completing the proof of Theorems \ref{main-intro} and \ref{main-intro2}.
In Section 9 we demonstrate, by example, how
$H^1_w$ is typically infinite-dimensional, and how an ILC tiling space
can be homeomorphic to an FLC tiling space without being topologically
conjugate to any FLC tiling space.  Finally, in Section 10 we prove 
Theorem \ref{Any-map-theorem} and its ILC analogue.

\paragraph*{Acknowledgments}
Work of the second author is partially supported by NSF Grant DMS-1101326.
We thank Johannes Kellendonk and Christian Skau for helpful discussions.

\section{Tilings and Delone sets}

In this section, we define the main objects which will be studied in
the next parts of the paper, namely tilings and (labeled) Delone
sets. For some applications, such as  constructing transversals, it is
more convenient to work with Delone sets. For others, such as
pattern-equivariant cohomology,  it is more convenient to work with
tilings. However, the two constructions are essentially equivalent.

\begin{defin}
  A Delone set of $\R^d$ is a subset $\Lambda \subset \R^d$ which is
  \emph{uniformly discrete} and \emph{relatively dense}, in the sense
  that there exist respectively $r > 0$ such that any two points of
  $\Lambda$ are at least at distance $r$ apart; and $R> 0$ such that
  any ball of radius $R$ in $\R^d$ intersects $\Lambda$.
 
  A \emph{labeled Delone set} is (the graph of) a map $\ell: \Lambda
  \rightarrow X$ from a Delone set to a compact metric space $X$,
  called the space of labels.  Its elements are pairs $(\lambda,
  \ell(\lambda))$ where $\lambda$ is a point of $\R^d$ and
  $\ell(\lambda)$ is its label.  By abuse of terminology, we still
  call a labeled Delone set simply $\Lambda$.
\end{defin}

\begin{defin}
  A tile is a compact subset of $\R^d$ that is homeomorphic to a
  closed ball of positive radius. For simplicity, we will assume that
  a tile is a convex polytope.
Tiles that are translates of one another
are said to be equivalent. We pick a distinguished representative
from each equivalence class, called a {\em prototile}. 
The set of prototiles can
  be given a topology induced by the Hausdorff distance on tiles.
  Given a compact set of prototiles $\ron A$, a tiling with tiles in
  $\ron A$ is a set $T = \{t_i\}_{i \in I}$ where the $t_i$ are tiles
  whose class is in $\ron A$, the union of the $t_i$ is all of $\R^d$,
  and any two distinct tiles only intersect on their boundary.

  Similarly, we can define a labeled tiling as the graph of a map from
  $T$ to a compact metric space of labels $X$. The set of labeled
  prototiles is then a compact subset of $\ron A \times X$. If $X$ is 
infinite, we also assume
that the shapes and positions of the prototiles are such that 
convergence in $X$ implies convergence in $\ron A \times X$.
\end{defin}

Given a labeled Delone set, one can construct a tiling by looking at
the Voronoi cells of the points of the Delone set, \emph{i.e.}, at the regions
closer to a particular point than to all others. Conveniently, Voronoi
cells are always convex polytopes.  One can also take the dual tiling
to the Voronoi tiling, thereby obtaining a tiling whose vertices are
the points of the original Delone set. 

Conversely, given a tiling, we can construct a Delone set as follows.
To each labeled tile $(t,l)$ we associate a distinguished point $x_t$,
called a puncture (\emph{e.g.},  at the center of mass), and consider the
labeled point $(x_t, ([t],l))$ where $([t],l) \in \ron A \times X$.
Therefore, problems worded in terms of tilings can be re-worded in
terms of Delone sets, and vice versa.  We use the generic term
``tiling space'' to refer either to a space of tilings or a space of
Delone sets.

Note that if we start with a tiling by arbitrary shapes, construct a 
Delone set, and then construct the corresponding Voronoi tiling, we obtain
a tiling by convex polytopes. In fact, this ``Voronoi trick'' induces an 
MLD equivalence between the original tiling space and a space of tilings by 
convex polytopes. For purposes of understanding tiling spaces and homeomorphisms
up to MLD equivalence, our assumption that all tiles are convex polytopes does
not sacrifice any generality. 

A \emph{patch} or \emph{local configuration} is the restriction of a
tiling or Delone set to a bounded region. Formally, if $\Lambda$ is a
(labeled) Delone set, $\Lambda \cap B(x,R)$ is the restriction of
$\Lambda$ to the ball of radius $R$ and center $x$.  For a tiling, $T
\cap B(x,R)$ is the set of all (decorated) tiles which intersect
$B(x,R)$.

The set of tilings (or Delone sets) can be given a metric
topology. For Delone sets, it coincides with vague convergence of
measures, when $\Lambda$ is identified with the sum of Dirac measures
with support $\Lambda$. Note that the Delone set property implies that
whenever $\Lambda_n$ converges to $\Lambda$ vaguely, any point
$x$ in $\Lambda$ corresponds to a point $x_n \in \Lambda_n$ for $n$
large enough, and $x_n \ra x$.  This justifies the following
definition for convergence of labeled Delone sets.
\begin{defin}
  Let $\Lambda$, $\Lambda'$ be two labeled Delone sets, with the same
  space of labels $X$.  We say that the distance between $\Lambda$ and
  $\Lambda'$ is less than $\eps$ if there exist a one-to-one map
  $\phi: \Lambda \cap B(0,\eps^{-1}-\eps) \ra \Lambda' \cap
  B(0,\eps^{-1})$ and a one-to-one map $\phi': \Lambda' \cap
  B(0,\eps^{-1}-\eps) \ra \Lambda \cap B(0,\eps^{-1})$ which are
  inverse of each other (where defined), and such that
 \[
  \forall x, x', \ \nr{\phi(x)-x} + \nr{\phi'(x')-x'} + D(\ell(x), \ell(\phi(x)) + D(\ell(x'),\ell(\phi'(x')) < 2 \eps,
 \]
 where $D$ is the distance on the space of labels.
\end{defin}

That's a very complicated definition! However, the idea behind it is 
simple. 
Two Delone sets are
close if on a large ball (of size $1/\eps$) 
around the origin, the point-sets are very
close; there is a one-to-one correspondence between points,
corresponding points are $\eps$-close in $\R^d$ and their labels are 
$\eps$-close in $X$.
Similarly, two tilings $T$ and $T'$ are close if on a large ball
around the origin, they are nearly the same, in that there
is a one-to-one correspondence between tiles, and corresponding tiles
are close both in shape for the Hausdorff distance and in label.

\begin{defin}
  A tiling $T$ or Delone set $\Lambda$ has \emph{finite local
    complexity} (or FLC) if it has are only finitely many local
  configurations of a given size, up to translation.  Otherwise, it
  has infinite local complexity (or ILC).
\end{defin}

The topology for FLC tilings is induced by a much simpler metric: 
$\Lambda$ and $\Lambda'$
are within distance $\eps$ if there exists $x,x'$ of norm less than
$\eps$ such that $(\Lambda-x) \cap B(0,\eps^{-1}) = (\Lambda'-x') \cap
B(0,\eps^{-1})$.

The group $\R^d$ acts on tilings and Delone sets by translation:
if $\Lambda$ is a Delone set, we write 
\[
 \Lambda - x := \bigl\{ \lambda - x \ : \ \lambda \in \Lambda \}.
\]
Similarly if $T = \{t_i\}$ is a tiling (\ie a collection of tiles), $T-x$ is the
collection $\{t_i - x\}$.

Given a Delone set $\Lambda$ (or a tiling), one can define its \emph{hull}
as the closure
\[
 \Omega_\Lambda := \overline{ \{ \Lambda - x \ ; \ x \in \R^d \} }.
\]
Elements in the closure can still be interpreted as Delone sets
(resp.\ tilings). This space supports an action of $\R^d$ by
translations, and is a compact dynamical system.  Our focus will
be on \emph{minimal} spaces, meaning that every
$\R^d$-orbit is dense.
The well-known Penrose tiling is an example of a minimal, aperiodic tiling space.

\section{Tiling equivalences and transversals }

We defined so far two families of tiling spaces: FLC tiling spaces and
ILC tiling spaces.  We now define a family of maps between tilings
spaces, which can be seen as equivalence between the dynamical
systems.  FLC tiling spaces have additional structure, so it makes
sense to define a family of more rigid maps between FLC tiling spaces.

\begin{defin}
 A \emph{factor map} $\phi: \Omega \ra \Omega'$ is a continuous onto map that
 satisfies
 \[
   \phi (T-x) = \phi(T) - x
 \]
for all $x \in \R^d$. 
 In particular, it sends orbits to orbits.
 A \emph{topological conjugacy} is a factor map that is a homeomorphism.
\end{defin}

\begin{defin}\label{def:local-map}
 Let $\Omega$ and $\Omega'$ be FLC tiling spaces. A \emph{local map} from $\Omega$ to $\Omega'$ a continuous map which satisfies
 \(\forall r> 0 , \ \exists R > 0, \ \forall T,T' \in \Omega\),
 \[
  \big( T \cap B(0,R) = T' \cap B(0,R) \big) \Rightarrow \big( \phi(T) \cap B(0,r) = \phi (T') \cap B(0,r) \big).
 \]
 A \emph{local derivation} is a local map that is also a factor map.
 A \emph{mutual local derivation} (or MLD map) is a local derivation
 that is a homeomorphism, and whose inverse is a local derivation. In
 particular, it is a conjugacy.
\end{defin}


Transversals can be defined for all tiling spaces, but we will
require more structure from transversals of FLC tiling spaces.

\begin{defin}\label{def:canonical-transversal}
 Let $\Omega$ be a space of Delone sets. The \emph{canonical transversal} of $\Omega$ is the set
 \[
   \Xi = \overline{ \{ \Lambda - \lambda \ ; \ \lambda \in \Lambda \}} .
 \]
 That is, it is the set of all Delone sets in $\Omega$ which have a point at the origin.
\end{defin}

\begin{defin}
  Let $\Omega$ be a tiling space (or a space of Delone sets). A
  transversal for $\Omega$ is given by a factor map $\ron D: \Omega
  \ra \Omega'$ to a space of Delone sets, and is defined by
 \[
  \Xi_{\ron D} = \ron D^{-1} (\Xi'),
 \]
 where $\Xi'$ is the {\em canonical} transversal of $\Omega'$.\footnote{
 Note that in particular, if $\Omega$ is a space associated with a
 Delone set, its \emph{canonical} transversal in the sense of definition~\ref{def:canonical-transversal}
 is a transversal in the present sense; however, there exist many transversals besides the canonical one.}

 When $\Omega$ has FLC, we define an \emph{FLC-transversal}
 (or simply transversal when it is clear from the context)
 identically, except that the factor map $\ron D$ is assumed
 to be a local derivation.
\end{defin}

We call such a 
factor map $\ron D$ a ``pointing rule'' for the tilings of $\Omega$.

\begin{ex}
  As a typical example, if $\Omega$ is a tiling space with convex
polyhedral tiles, let $\ron D
  (T)$ be defined as the Delone set of all barycenters of tiles in
  $T$.  Then the associated transversal $\Xi_{\ron D}$ is often called
  in the literature ``the canonical transversal of~$\Omega$''.  When
  $\Omega$ has FLC, this transversal is an FLC-transversal.
\end{ex}

This last example can be generalized considerably. Even if the tiles 
are not polyhedra, we can pick a distinguished point in each prototile
(sometimes called a {\em puncture}) and let $\ron D(T)$ be the set of all
punctures of tiles in $T$. We could also restrict attention to the punctures 
of just one kind of tile, or to tiles that meet additional conditions
(e.g. the punctures of those ``$A$'' tiles that are surrounded by ``$B$''
tiles). 

Note that the definitions of pointing rules and transversals 
do not require $\ron D$ to
be a local derivation, just a factor map. The additional flexibility
provided by allowing such general pointing rules will prove to be 
essential.

Tiling spaces have a local product structure in
which one of the factors is given by $\R^d$, the group acting by
translations. The following result states a known fact: tiling spaces
(with or without FLC) are examples of foliated spaces (also called
laminations).
\begin{prop}\label{prop:lamination}
 Let $\Omega$ be a minimal aperiodic tiling space, with or without
 FLC. Then for all $T$, there exists a transversal $\Xi$ containing
 $T$ and $C > 0$ such that the map
 \[
  \Xi \times B(0,C) \ra \Omega \ ; \ (T',x) \mapsto T'-x
 \]
 is a homeomorphism onto its image (in particular its image is open).
\end{prop}

\begin{proof}
 Assume $\Omega$ is a space of Delone sets, and $0 \in T$.
 Let $r, R$ be the Delone constants. Then, if $\Xi$ is the canonical
 transversal and $C = r/2$, it is easy to check the conclusions of
 the theorem.
 If $0 \notin T$, let $\Xi$ be the transversal given by the pointing
 rule $\ron D (T') = T'-x$, where $x$ is any given point of $T$.
\end{proof}

In the FLC case, we have more rigidity.
\begin{prop}[\cite{BBG06}]
 Let $\Omega$ be an aperiodic, minimal tiling space with finite local complexity. Then for all $T_0 \in \Omega$, there exist $\eps > 0$ and an FLC transversal $\Xi$ containing $T$ such that the map
 \[
  B(0,\eps) \times \Xi \longrightarrow \Omega \ ; \ (x,T) \longmapsto T-x
 \]
 is a homeomorphism onto its image. In addition, $\Xi$ is a Cantor set (compact, totally disconnected and with no isolated points).
\end{prop}

It implies in particular that for FLC, aperiodic and minimal tiling spaces, the path-connected components are exactly the $\R^d$-orbits.

Finally, local maps can be characterized in terms of FLC transversals.
\begin{prop}\label{prop:local}
 Let $f: \Omega \ra \Omega'$ be a continuous map between FLC tiling
 spaces.
 It is a local map if and only if the image of any FLC transversal
 is included in an FLC transversal
\end{prop}

\begin{proof}
 Let $f$ be a local map between FLC tiling spaces.
 Let $\Xi$ be an FLC transversal. 
 By definition, there is a local pointing rule $\ron D$ such
 that for all $T$, $\ron D (T)$ is a Delone set locally derived from $T$.
 Let $R_1$ be the constant of locality of $\ron D$ in the sense that whenever two tilings agree up to distance $R_1$, their images by $\ron D$ agree up to distance $1$.
 Now, we claim that the set
  \(
   \bigl\{ T \cap B(0,R_1) \ : \ T \in \Xi \bigr\}
  \)
 is finite.
 Assume, in search of a contradiction that it is infinite.  By
 compactness of $\Omega$, there is a sequence $(T_n)_{n \in \N}$ of
 tilings in $\Xi$ which all disagree within distance $R_1$ but
 converges in $\Omega$, and is therefore a Cauchy sequence.  So for
 all $\eps > 0$, there exists $n,m$ big enough and $x,y \in \R^d$ such
 that $T_n - x$ and $T_m - y$ agree up to $R_1+2\eps$ with $x$ and $y$
 smaller than $\eps$. Note also that $x \neq y$ since $T_n$ and $T_m$
 must disagree within radius $R_1$.  Therefore, $0 \in \ron D(T_n)$
 (because $T_n \in \Xi$) and $0 \in \ron D (T_n - x + y)$ (because
 this tiling agrees up to radius $R_1$ with $T_m \in \Xi$).  Since
 $\ron D$ is a factor map, both $0$ and $x - y$ belong to $\ron
 D(T_n)$. Now, $\nr{x - y} < 2\eps$ which can be picked arbitrarily
 small. This contradicts the Delone property of $\ron D(T_n)$.

 We let $P_i$ be the finitely many local configurations (of the form $T_i \cap B(0,R_1)$ for some $T_i \in \Xi$) such that
 \(
  0 \in \ron D(T) \Leftrightarrow \exists i, \ T \cap B(0,R_1) = P_i
 \).
 We let $\Xi_i$ be the sub-transversal of $\Xi$ defined by $T \in \Xi_i \Leftrightarrow T \cap B(0, R_1) = P_i$. It is enough to show that for all $i$, $f(\Xi_i)$ is a FLC transversal (as a finite union of FLC transversals is itself a FLC transversal).
 We therefore assume, to simplify notation and without loss of generality, that $i=1$ and simply write $P = P_i$, $\Xi = \Xi_i$.
 Let $R_2$ be the constant of locality for the map $f$ such that for all tilings $T_1, T_2$,
 \[
  T_1 \cap B(0,R_2) = T_2 \cap B(0,R_2) \Rightarrow
    f(T_1) \cap B(0,1) = f(T_2) \cap B(0,1).
 \]
 By FLC, there are finitely many patches $Q_1, \ldots, Q_j$ of radius at least $R_2$ which extend $P$,
 that is $Q_i \cap B(0,R_1) = P$.
 Therefore, by locality of $f$, there are finitely patches of the form
 $Q'_i = f(T) \cap B(0,1)$, for $T \in \Xi$.
 Define the following pointing rule on $\Omega'$:
 \[
 \begin{split}
  0 \in \ron D' T' & \Longleftrightarrow \exists T \in \Xi, \ 
  T' \cap B(0,1) = f(T) \cap B(0,1) \\
     & \Longleftrightarrow \exists i, \ T' \cap B(0,1) = Q'_i.
 \end{split}
 \]
 This pointing rule defines an FLC transversal $\Xi'$ which
 contains $f(\Xi)$.
 
 Conversely, assume $f$ is continuous, and the image of any
 FLC transversal is included in an FLC transversal.
 Let $T_0 \in \Omega$.
 Define a pointing rule on $\Omega$ by $0 \in \ron D (T')$
 if and only if $T' \cap B(0,1) = T_0 \cap B(0,1)$.
 Then, the image by $f$ of $\Xi$ is included in an FLC transversal
 $\Xi'$, associated with a pointing rule $\ron D'$.
 The map $f$ induces a continuous map $\Xi \ra \Xi'$. It is
 uniformly continuous, and given how the topology restricts
 on FLC transversals, uniform continuity reads: 
  $\forall r > 0, \ \exists R > 0, \ \forall T_1, T_2 \in \Xi,$
 \begin{displaymath}
  \bigl( T_1 \cap B(0,R) = T_2 \cap B(0,R) \bigr) \Longrightarrow
    \bigl( f(T_1) \cap B(0,r) = f(T_2) \cap B(0,r) \bigr),
 \end{displaymath}
 which is precisely the locality condition for $f$.
\end{proof}

\section{Tiling cohomologies}\label{sec:cohomology}

There are several cohomology theories for a tiling space $\Omega$,
which mostly yield isomorphic cohomology groups.
At first, we only consider \emph{aperiodic, minimal, FLC} tiling
spaces, and will later in this section discuss cohomology for non-FLC
tiling spaces.

To begin with, we may consider the \v Cech cohomology
$\check H^*(\Omega; A)$ with values in an Abelian group $A$.
We will restrict our discussion to $A = \R$ or $\R^d$.
This is computed from the
combinatorics of open covers of $\Omega$, and is manifestly a
homeomorphism invariant. The trouble with \v Cech cohomology is that
it is difficult to relate the cohomology of the tiling {\em space} to
properties of individual {\em tilings}.

Next there is the (strong) {\em pattern-equivariant} cohomology of $\Omega$
with values in $A$, denoted $H^*_{PE,s}(\Omega; A)$.  Let $T$ be an FLC
tiling of $\R^d$ whose orbit is dense in $\Omega$. The tiling $T$ gives a
decomposition of $\R^d$ as a CW complex. The 0-cells are the vertices
of the tiling, the 1-cells are the edges, and so on.

A $k$-cochain $\alpha_k$ on $T$ is an assignment of an element of $A$ to each
$k$-cell of $T$.
The coboundary of a cochain is taken in the sense of cellular cohomology
for the $CW$-decomposition of $\R^d$ given by $T$.
The cochain $\alpha$ is called (strongly) {\em pattern
  equivariant} (or PE) if there exists a distance $R$ such that the
value of $\alpha$ on each cell depends only on the patch of radius $R$
around the cell. That is, if there are cells $c_1$ and $c_2$ centered
at points $x_1$ and $x_2$, and if the patch of radius $R$ centered
around $x_1$ matches the patch of radius $R$ centered at $x_2$, then
$\alpha(c_1)=\alpha(c_2)$. It is not hard to see that the coboundary
of a PE cochain is PE (with perhaps a slightly different radius), so
the PE cochains form a differential complex.  $H^*_{PE,s}(T; A)$ is the
cohomology of this complex. It is known~\cite{Kel03,CS06}
that $H^*_{PE,s}(T; A)$ is canonically isomorphic to
$\check H^*(\Omega; A)$ for FLC tiling spaces.
Thus, if $\Omega$ is an aperiodic, FLC, minimal tiling
dynamical system, then its \v Cech cohomology may be computed from the
pattern equivariant cohomology of \emph{any} tiling in $\Omega$. 
Moreover, PE cochains on any one tiling $T$ extend by continuity to PE 
cochains on all other tilings in $\Omega$.
We can thus speak of the complex of PE cochains on $\Omega$,
rather than on a specific tiling, and denote 
the cohomology $H^*_{PE,s}(\Omega; A)$. 

When $A$ is $\R$ or $\R^d$, PE cohomology can also be defined using
differential forms. (Indeed, this is how PE cohomology was first
introduced; see~\cite{Kel03, KP06} for the definition and an isomorphism
with \v Cech cohomology.) One considers functions and differential
forms on $\R^d$ with values in $A$, which are PE with respect to a tiling $T$,
 and constructs the de Rham complex.
This yields a cohomology $H^*_{dR,s}(\Omega; A)$ 
isomorphic to $H^*_{PE,s}(\Omega; A)$, the version of PE cohomology based
on cellular cochains.
Indeed, given a PE $k$-form, we can construct a PE
cochain by integrating the form over $k$-cells.  Conversely, given a
PE cochain we can obtain a PE $k$-form whose integral is that
cochain, \emph{e.g.}, by linearly interpolating across tiles.  These
operations give explicit isomorphisms between the form-based and
cochain-based PE cohomologies.

In this case, we can also consider {\em weakly} pattern equivariant
cochains or forms (weakly PE, or w-PE), which are defined to be
uniform limits of strongly PE cochains or forms.\footnote{For
  $C^\infty$ forms, we mean by this convergence for all the semi-norms
  $\|f\|_{(k)} = \|f^{(k)}\|_\infty$. 
That is, we assume that the
derivatives of the forms of all orders are uniform limits of PE
forms. For PE-cochains, we mean
  convergence for the norm $\|\alpha\| = \sup \{ \alpha(c) \ : \ c
  \text{ chain} \}$.} 
This gives a cohomology theory that we call {\em weak} PE
cohomology, denoted $H^*_{PE,w}$ or $H^*_{dR,w}$.  This usually is
{\em not} isomorphic to the \v Cech cohomology of $\Omega$. Rather, it
corresponds to the ``dynamical'' Lie-algebra cohomology of $\R^d$ with
values in the space of continuous functions on $\Omega$~\cite{KP06}.
The group $H^*_{PE,w}$ is usually infinitely generated, even for very
simple tiling spaces, as will be illustrated in
Section~\ref{sec:examples}.

When the tilings have infinite local complexity, then it is
unreasonable to expect patches to be identical, and strong pattern
equivariance is not usually well-defined. However, a version of weak
pattern equivariance always makes sense, as we will see.
We present this using an adaptation of groupoid cohomology.
For the purpose of this article and to avoid lengthly discussion,
we focus on the group $H^1$, with coefficients in $\R^d$.

Let $T$ now be a given tiling in a minimal, aperiodic tiling space $\Omega$,
which may or may not have finite local complexity.

\begin{defin}
 If $T$ has finite local complexity,
 a continuous function $f: \R^d \ra \R^d$ is called
 \emph{strongly $T$-equivariant} (or simply \emph{pattern
 equivariant} or PE) if there exists a radius
 $R$ such that $f(x) = f(y)$ whenever $T \cap B(x,R) = T \cap B(y,R)$ (up
 to translation).
 
 Given a tiling $T$ with or without finite local complexity, a
 continuous function $f$
 is called \emph{weakly pattern equivariant} (weakly PE or w-PE)
 if $f(x_n) \ra f(x)$ whenever $x_n$ is a sequence
 satisfying $T-x_n \ra T-x$ in the tiling space of $T$. When
 $T$ has FLC, weakly PE functions are uniform limits of strongly
 PE functions.
\end{defin}

\begin{defin}
\label{def:dynam-cocycle}
 Let $T$ be an aperiodic tiling with minimal tiling space.
 A $T$-equivariant weak (resp.\@ strong\footnote{Under the additional assumption that $T$ has FLC}) dynamical 
$1$-cocycle with values in $\R^d$ is a
 function $\beta: \R^d \times \R^d \ra \R^d$ which satisfies
 \begin{itemize}
  \item \emph{Cocycle condition:} $\beta(x,v) + \beta(x+v,w) = \beta(x,v+w)$;
  \item \emph{Pattern equivariance:} for all $v$, the function
  $x \mapsto \beta(x,v)$ is weakly (resp.\@ strongly) pattern equivariant.
 \end{itemize}
 This cocycle is a \emph{weak (resp.\@ strong) $1$-coboundary} if there exists a
 continuous, weakly (resp.\@ strongly) pattern-equivariant function
 $s: \R^d \ra \R^d$ such that for all $(x,v) \in \R^d \times \R^d$,
 $\beta(x,v) = s(x+v) - s(x)$.
\end{defin}

Note that strong $1$-cocycles can be weak coboundaries, even when
they are not strong coboundaries.

\begin{rem}
 The weak pattern equivariance condition is a continuity condition
 for an appropriate topology: it is possible to embed
 $\R^d \times \R^d$ in $\Omega \times \R^d$ \emph{via} the map
 $(x,v) \mapsto (T-x, v)$.
 By minimality, this embedding has dense image.
 A function defined on $\R^d \times \R^d$ will then extend continuously
 to a function on $\Omega \times \R^d$ if and only if it is weakly
 PE.
 Similarly, a PE $0$-cochain can be seen as a continuous function on
 $\Omega$.
\end{rem}

As a consequence of this remark, it is possible to describe these
PE cochains as follows:
\begin{itemize}
 \item a weak PE $1$-cocycle $\alpha: \R^d \times \R^d$
 corresponds to a continuous function (still called $\alpha$ by abuse of notation)
 $\alpha: \Omega \times \R^d \ra \R^d$ such
 that $\alpha(T,x) + \alpha(T-x,y) = \alpha(T,x+y)$;
 \item a weak $0$-cochain corresponds to a continuous function $s: \Omega \ra \R^d$;
 \item the co-boundary of a $0$-cochain in this picture is defined by 
 \(
  \delta s (T,v) = s(T-v) - s(T)
 \).
\end{itemize}
This shows that the groups of cocycles and coboundary do not
depend on the particular $T$ used to define pattern-equivariance.

The notion of strong pattern equivariance is a bit more delicate.
Strongly pattern equivariant functions on $\R^d$ correspond
to continuous functions on $\Omega$ that are \emph{transversally
locally constant}, in the sense that their restriction to any FLC
transversal is locally constant. (This definition requires in particular
that $T$ has FLC.)

\begin{defin}
  Given a minimal, aperiodic tiling space $\Omega$ with or without FLC,
  its dynamical $1$-cohomology group $H^1_{d, w} (\Omega;\R^d)$
  is the quotient of the space of weak dynamical $1$-cochains by
  the space of weak coboundaries.
  
  When the tiling space has FLC, we can define in addition
  $H^1_{d,s}$ as the quotient of the strong dynamical $1$-cochains
  by the strong coboundaries.
\end{defin}

It will be convenient to define cohomology by using slightly smaller
groups.
Let $\Omega$ be a tiling space and $\Xi$ a transversal.
Let $\ron D$ be the associated pointing rule.
Let $T$ be a fixed tiling, so that $\ron D (T) \subset \R^d$ is a Delone set.
Let $(\R^d \times \R^d)_{|\Xi}$ be the set of all $(x,v) \in \R^d \times \R^d$
such that $x \in \ron D(T)$ and $x+v \in \ron D(T)$.
We call \emph{restriction to the transversal $\Xi$} the restriction of
$1$-cochains to this subset. Similarly, PE $0$-cochains can be restricted
to $\ron D (T)$.
It is important to note that, while we restrict cochains a certain subset
of $\R^d \times \R^d$ determined by $\ron D(T)$, the pattern-equivariance
condition for these cochains is still defined in terms of $T$ and \emph{not}
in terms of $\ron D(T)$. Therefore, the corresponding cohomology group
will describe the topology of $\Omega$, and not the topology of
the tiling space of $\ron D (T)$.
When dealing with strongly PE cochains, we require that $\Xi$ is an FLC
transversal, so that $\ron D (T) \subset \R^d$ has FLC.

The following result shows that we can use either cochains on
$\R^d \times \R^d$ or restricted cochains to represent an element
in $H^1_{d,s/w}(\Omega; \R^d)$

\begin{prop}\label{allthesamecohomology}
 Let $T$ be a tiling with or without FLC , such that its tiling
 space $\Omega$ is minimal and aperiodic.
 Let $\Xi \subset \Omega$ be a transversal.
 We define
  $\mathcal Z^1_{w,\Xi}(\Omega;\R^d)$ to be the set of weakly PE dynamical $1$-cochains
  on $(\R^d \times \R^d)_{|\Xi}$ (\emph{restricted} cochains), and
  $\mathcal B^1_{w,\Xi}(\Omega;\R^d)$ the set of weakly PE restricted coboundaries.
  
 Then restriction of $1$-cochains on $\Xi$ induces an
 isomorphism 
 \[
  H^1_{d,w}(\Omega; \R^d) \lra \mathcal Z^1_{w,\Xi}(\Omega;\R^d) / \mathcal B^1_{w,\Xi}(\Omega;\R^d).
 \]
 
 Whenever $T$ has FLC, and $\Xi$ is an FLC transversal,
 the same statement holds for the strong cohomology groups, \emph{i.e.}, the
 group $H^1_{d,s}(\Omega;\R^d)$ is isomorphic \emph{via} the restriction map
 to the quotient of strongly PE restricted cochains by strongly PE restricted
 coboundaries.
\end{prop}

\begin{proof}
 We show the isomorphism for weak cochains. The strongly PE case is
 similar.
  Let $T$ be fixed, for which cochains are $T$-equivariant.
  Without loss  of generality, assume $0 \in \ron D (T)$.
  Clearly a $1$-cochain (resp.\@ coboundary) defined on $\R^d \times
  \R^d$ restricts to a cochain (resp.\@ coboundary). So there is a
  well-defined map from the associated cohomology groups.
 
  To show that the map is one-to-one and onto, we essentially follow the
  ideas of~\cite[Lemma 2.6]{Kel08}. The results were stated in the FLC
  case, but extend to the general case.
 
  The map is onto: let $\beta$ be a restricted $1$-coboundary.
  The idea, as in~\cite[Lemma 2.6, (3)]{Kel08}, is to build a function $\phi: \R^d \ra \R^d$  such that $\phi(\lambda) = \beta(0,\lambda)$ for all $\lambda \in \ron D (T)$.
  The function $\phi$ is obtained by interpolating first $\lambda \mapsto \beta(0,\lambda)$ by a locally constant function (which is constant on the interior of each Voronoi cell of $\ron D(T)$), and then smoothing it by a compactly supported positive function of integral~$1$.
  The cochain $\alpha$ is given by $\alpha(x,v) := \phi(x + v) - \phi(x)$.
  
  We detail the construction now.
  First, remark that because of the continuity of the pointing rule, if $T-x$ and $T-y$ are close to each other, then the Delone sets $\ron D(T) - x$ and $\ron D(T)-y$ are also close to each other.
  Besides, assume $x,y \in \ron D(T)$, let $t_x, t_y$ be the Voronoi cells of $x$ and $y$, \emph{i.e.}, the set of points which are closer from $x$ (resp.\ $y$) than any other point in $\ron D(T)$.
  If $\ron D(T) - x$ and $\ron D(T) - y$ are very close from eath other, then $t_x - x$ and $t_y - y$ are also close to each other for the Hausdorff distance and therefore, the indicator functions $\chi_{t_x - x}$ and $\chi_{t_y - y}$ are close for the $L^1$ norm.
  
  Starting from the cochain $\beta$, define $\psi = \sum_{u \in \ron D(T)} \beta(0,u) \chi_{t_u}$.
  For $x, v \in \R^d$ consider the function $\tilde \alpha(x,v) = \psi(x+v) - \psi(x)$, so that the cocycle $\alpha$ is given by $\alpha(\cdot, v) = \tilde \alpha(\cdot, v) * \rho$ for $\rho$ a smooth positive function of integral 1, supported on a small ball around the origin.
  It is easy to check that $\alpha$ is a cochain, and if the support of $\rho$ is small with respect to the uniform discreteness constant of $\ron D(T)$, then $\alpha$ restricts to $\beta$.
  
  We now show pattern equivariance. Fix $v \in \R^d$.
  \[
   \tilde \alpha(\cdot, v) = \sum_{u \in \ron D(T)} \beta(0,u) \bigl[ \chi_{t_u - v} - \chi_{t_u} \bigr].
  \]
  Assume $x, y \in \R^d$ are such that $T-x$ and $T-y$ are very close, to the point that $T-x-v$ and $T-y-v$ are also very close.
  Then, the Voronoi tilings around these points are also very close, and by definition of closeness for ILC tilings, there is a one-to-one correspondence between Voronoi tiles in a neighborhood of $x$ and Voronoi tiles in a neighborhood of $y$.
  A similar statement holds for tiles neighbor to $x+v$ and $y+v$.
  Therefore, for any pair of Voronoi tiles $(t_x, t'_x)$ with $t_x$ within distance $1$ of $x$ and $t'_x$ within distance $1$ of $x+v$, there exists a pair $(t_y, t'_y)$ with $t_y$ close to $y$ and $t'_y$ close to $y+v$. In addition, $t_x - x$ and $t_y - y$ are Haudsorff-close, and their indicator functions are $L^1$-close (similarly, the indicator functions of $t'_y$ and $t'_x$ are $L^1$-close).
  Now, in a neighborhood of $x$,
  \[
  \begin{split}
   \tilde \alpha(\cdot, v) & = \sum_{(t_x, t'_x)}
    \bigl( \beta(0,p_{t'_x}) - \beta(0,p_{t_x})  \bigr) \chi_{t_x}\chi_{t'_x - v} \\
    & = \sum_{(t_x, t'_x)} \beta(p_{t_x}, p_{t'_x} - p_{t_x})
    \chi_{t_x}\chi_{t'_x - v}
  \end{split}
  \]
  where the sum is taken over all pairs of tiles $(t_x, t'_x)$ as defined above, and $p_t$ is the point $\ron D(T)$ defining the Voronoi tile $t$.
  A similar formula holds for the value of $\tilde \alpha (\cdot, v)$ in a neighborhood of $y$.
  
  Now, if $T-x$ and $T-y$ are close enough, each pair $(t_x, t'_x)$ in the sum above corresponds to a pair of tiles $(t_y, t'_y)$ sitting respectively around $y$ and $y+v$. Furthermore, for each $t_x$, $T-p_{t_x}$ and the corresponding $T-p_{t_y}$ are close in the tiling distance. In addition, the vectors $p_{t'_x} - p_{t_x}$ and $p_{t'_x} - p_{t_y}$ are close to each other in $\R^d$.
  Therefore, since $x \mapsto \beta(x,u)$ is PE for any fixed $u$, and $\beta$ is continuous of two variables, then for each corresponding pairs $(t_x, t'_x)$ and $(t_y, t'_y)$, the indicator functions
  \[
   \bigl( \beta(p_{t_x}, p_{t'_x} - p_{t_x}) \bigr) \chi_{t_x}\chi_{t'_x - v} 
   \text{ and }
   \bigl( \beta(p_{t_y}, p_{t'_y} - p_{t_y}) \bigr) \chi_{t_y}\chi_{t'_y - v}
  \]
  are close in the $L^1$ norm.
  We deduce, by properties of the convolution, that the values of the cocycle $\alpha( \cdot, v) = \tilde \alpha(\cdot, v) * \rho$ at $x$ and $y$ are very close, provided $T-x$ and $T-y$ are close enough.
  Therefore, $\alpha$ is PE.
   
  So $\alpha$ is a $1$-cochain defined on $\R^d \times \R^d$ which
  restricts to $\beta$.
 
  The map is one-to-one: this is a similar argument on coboundaries
  using the ideas of~\cite[Lemma 2.6, (1)]{Kel08}.
\end{proof}

\begin{thm}\label{thm:allsame}
  For a FLC, minimal, aperiodic tiling space, the weak cohomology
  groups $H^1_{PE,w} (\Omega; \R^d)$, $H^1_{dR,w}(\Omega; \R^d)$ and
  $H^1_{d,w}(\Omega;\R^d)$ are isomorphic.  Similarly, the strong
  cohomology groups $H^1_{PE,s} (\Omega; \R^d)$, $H^1_{dR,s}(\Omega;
  \R^d)$ and $H^1_{d,s}(\Omega;\R^d)$ are isomorphic.
\end{thm}

\begin{proof}
 Let $T$ be fixed. When we mention pattern equivariance, it will be with
 respect to this $T$.
 
 We prove isomorphism of the dynamical cohomology (represented by dynamical cocycles
 as defined in Definition~\ref{def:dynam-cocycle}), and PE cellular cohomology
 (represented by cellular cochains of $\R^d$, which are PE with respect to $T$).
 An isomorphism between PE cellular and PE de Rham is given in~\cite{BK10}, both
 for the weak and strong cohomologies.
 
 The tiling $T$ gives a $CW$ decomposition of $\R^d$ in which an oriented edge $e$
 can be described by an origin $x$ and a vector $v$, and we write $e=(x,v)$
 (so that $-e = (x+v, -v)$).
 Given a dynamical $1$-cocycle $\alpha: \R^d \times \R^d \ra \R^d$,
 define a cellular $1$-cochain by $\beta(e) = \alpha(x,v)$.
 It is an exercise to show that $\beta$ is indeed a (cellular) cocycle, and is
 PE (strongly when $\alpha$ is, weakly otherwise).
 
 Conversely, if $\beta$ is a cellular $1$-cocycle, let us define a
 dynamical $1$-cocycle restricted to the Delone set $\ron D(T)$ of all
 $0$-cells of $T$.  If $x$ and $x+v$ are in $\ron D(T)$, there is a
 finite path of edges with appropriate orientations $\{e_i\}_i =
 \{(x_i,v_i)\}_i$ such that $x_1=x$, $x_i+v_i = x_{i+1}$ and $\sum_i
 v_i = v$.  Then define $\alpha (x,v) = \sum_i \beta (e_i)$. By the
 cellular cocycle condition, this sum is independent of path and only
 depends on $(x,v)$.  The resulting map $\alpha$ is a restricted dynamical
 $1$-cocycle which is PE with respect to the first
 variable. It is weakly or strongly PE depending on $\beta$.  Note
 that, by the previous proposition, this restricted $1$-cocycle $\alpha$
can be interpolated to a PE $1$-cocycle on all of $\R^d \times \R^d$.
 
 An identification of the $0$-cochains and of the boundary map is checked
 similarly.
\end{proof}

\begin{rem}
  What was actually described, in the discussion above, was the
  groupoid cohomology of $\Omega \rtimes \R^d$ (or its reduction on a
  transversal $\Xi$) using two resolutions: one involving continuous
  functions and one involving transversally locally constant
  functions. It is an element of folklore (which was explained to us by
  John Hunton~\cite{Hunton-talk})
  that these cohomologies are isomorphic to the already investigated weak and strong
  cohomologies---not just in dimension~$1$. For example, one can see
  $\Omega$ as a model for the classifying space of the reduction of
  $\Omega \rtimes \R^d$ to $\Xi$.  However, caution
  should be exercised as there is more than one groupoid cohomology
  theory and not all theories give isomorphic cohomology groups for a
  groupoid and its reduction.  This is why explicit isomorphisms were
  given, at least for $H^1$.
 
 In light of Theorem \ref{thm:allsame}, we will henceforth usually 
 write $H^1_s(\Omega;\R^d)$ and $H^1_w(\Omega;\R^d)$
 for the strong and weak first cohomology groups, keeping in mind that
 at least in the FLC case, there are many pictures for describing
 these groups.
\end{rem}

The following result will be used a few times.  In one-dimensional
tiling spaces, it is
essentially the classical Gottchalk--Hedlund Theorem.  In higher
dimensions, it can be deduced from generalizations of the
Gottschalk--Hedlund Theorem. A direct proof in a related setting
(with FLC tilings and cochains that are assumed to be closed and
strongly PE) may be found in~\cite{KS14}.
\begin{prop}\label{prop:bounded-coboundary}
A  $1$-cocycle $\alpha$ is a $1$-coboundary if and only if it is bounded.
In the terminology of dynamical cohomology, a $1$-cocycle is a
$1$-coboundary if and only if for some $x$ (equivalently, for all
$x$), the map
$v \mapsto \alpha(x,v)$ is bounded.
\end{prop}

We end this discussion on cohomology by providing an explicit isomorphism
between the dynamical $H^1$ groups and the de~Rham $H^1$ groups which
Kellendonk defined using pattern-equivariant forms.

\begin{prop}\label{prop:dR-cochains}
  Let $\omega$ be a closed, pattern-equivariant $1$-form on $\R^d$.
  Let $f$ be a differentiable (\emph{a priori} not PE) function on
  $\R^d$ such that $df = \omega$.  Then, the function $\alpha(x,v) =
  f(x+v)-f(x)$ is a $1$-cocycle in the groupoid picture. It is
  weakly/strongly PE if $\omega$ is, and the map $\omega \mapsto
  \alpha$ induces an isomorphism in cohomology.
\end{prop}

\begin{proof}
Since $\alpha(x,v) = \int_x^{x+v} \omega$, $\alpha$ inherits the
pattern-equivariance properties of $\omega$. It is strongly PE (i.e.
determined exactly by the pattern of $T$ in a neighborhood of the
path from $x$ to $x+v$) if 
$\omega$ is, and weakly PE if $\omega$ is, and so defines an element
in $H^1_s$ or $H^1_w$, respectively. What remains is to 
construct an inverse, in cohomology, to the map 
$\omega \mapsto \alpha$.

Let $\rho$ be a compactly supported smooth and non-negative function of 
total integral 1. For any dynamical $1$-cochain $\alpha$, one can build
a cochain
 \[
  \beta(x,v):= \int_{\R^d} \alpha(x,v-s) \rho(s) ds,
 \]
$\beta$ is then a $1$-cochain that is differentiable with respect
 to its second variable.
 It is easily checked that $v \mapsto \alpha(x,v) - \beta(x,v)$ is bounded, and
 therefore $\alpha$ and $\beta$ lie in the same cohomology class by
Proposition \ref{prop:bounded-coboundary}.
Now, let $\phi(v) = \beta(0,v)$.
 Then $d\phi$ is a PE $1$-form, and $\beta \mapsto d\phi$ is an inverse in
 cohomology of the map $\omega \mapsto \alpha$: if one composes these maps
 \[
  \omega {\longmapsto} \alpha \stackrel{* \rho}{\longmapsto} \beta {\longmapsto} d\phi,
 \]
 then $d\phi$ is in the same cohomology class as $\omega$.
\end{proof}

\section{An invariant for homeomorphisms 
of FLC tiling spaces}

In this section we define the cohomological class $[h] \in 
H^1_s(\Omega, \R^d)$ associated to a homeomorphism $h: \Omega \to 
\Omega'$ of minimal, FLC tiling spaces. We begin by defining the
{\em fundamental shape class} of a tiling space. This class is defined
both for FLC tilings and for arbitrary tilings. 

If $T\in \Omega$ is an FLC tiling, then there is a strongly 
PE 1-cochain $\Delta x$
that  assigns to each oriented edge the displacement along that edge. This is 
manifestly closed, and represents a class in $H^1_s(\Omega, \R^d)$.
In terms of 
dynamical cocycles, $\Delta x$ corresponds to the function $\ron F(x,v)=v$.
Note that the cocycle $\ron F(x,v)=v$ is well-defined for 
arbitrary tiling spaces, not just for FLC tilings.

\begin{defin}\label{def:fund-shape}
  If $\Omega$ is an FLC tiling space, the fundamental shape class in 
$H^1_s(\Omega; \R^d)$ is the class represented by the 1-cochain $\Delta x$. 
If $\Omega$ is any tiling space, the fundamental shape class in 
$H^1_w(\Omega; \R^d)$ is the class represented by the cocycle $\ron F$.
\end{defin}
In both the FLC and ILC settings, we denote the fundamental shape
class as $[\ron F]$, or by $[\ron F(\Omega)]$ when we wish to be
explicit about the tiling space in question. It should be clear from
context whether we are referring to an element of $H^1_s$ or $H^1_w$.

Next we specialize to FLC tiling spaces and construct
an invariant in $H^1_s(\Omega; \R^d)$ that classifies homeomorphisms 
$\Omega \to \Omega'$ up to MLD equivalence and \qtran{}.

\begin{defin}\label{def:FLC-invariant}
Let $h: \Omega \to \Omega'$ be a homeomorphism of aperiodic, minimal and 
FLC tiling spaces. Then $[h] = h^* [\ron F(\Omega')] \in H^1_s(\Omega,\R^d)$ 
is called the (strong) {\em cohomology class of $h$}. 
\end{defin}

\begin{rem}
Although $[\ron F(\Omega')]$ is defined via the strongly PE cochain
$\Delta x$, $h^*(\Delta x)$ is not necessarily strongly PE, and so we cannot do
the pullback operation directly in $H^1_{PE,s}$. We must use the isomorphism
between $H^1_{PE,s}$ and $\check H^1$ to represent $[\ron F(\Omega')]$
as a class in $\check H^1(\Omega', \R^d)$, 
do the pullback in $\check H^1$, and then convert back to $H^1_{PE,s}$. 
\end{rem}

There is a setting, however, where the pullback \emph{can} be done directly. 
Suppose that $h$ sends transversals to transversals. That is, if 
$x$ is a vertex of $T$, then $h(T-x)$ has a vertex at the origin. Then
PE 0- and 1-cochains on $\Omega'$ pull back to PE 0- and 1-cochains 
on $\Omega$, and we have $[h] = h^*[\Delta x] = [h^*(\Delta x)]$ 
in $H^1_{PE,s}$.

\begin{ex} \label{example-shape-change} 
Let $\Omega$ be a non-periodic and minimal 
1-dimensional tiling space with two types of 
tiles, $A$ and $B$, of length $L_1$ and $L_2$. (\emph{E.g.}, $\Omega$ might
be a space of Fibonacci tilings, or Thue--Morse tilings.) 
Let $\Omega'$ be a tiling
space with two tiles, $A'$ and $B'$ of length $L_1'$ and $L_2'$, 
that appear in the exact same sequences
as the $A$ and $B$ tiles appear in $\Omega$. Let $h: \Omega \to \Omega'$
be constructed as follows. If $T$ is a tiling in $A$ and $B$ tiles 
where the origin is a fraction
$f$ of the way from the left endpoint of a tile $t$ to the right endpoint, then
$h(T)$ is a tiling in $A'$ and $B'$, with the same sequence, in which the
origin is a fraction $f$ of the way across the tile $t'$ corresponding to $t$.

For each tile type $t$, let $i_t$ be a 1-cochain that evaluates to 1 on each
$t$ tile and to 0 on all other tile types, and let $[i_t]$ be the corresponding
class in $H^1_{PE}$. Under $h^*$, $i_{A'}$ pulls back to
$i_A$ and $i_{B'}$ pulls back to $i_B$. 
On $\Omega'$, we have $\Delta x = L_1' i_{A'} + L_2' i_{B'}$, 
so $[h] = [h^* \Delta x] = L_1' [i_A] + L_2' [i_B]$. The difference
between $[h]$ and $[\ron F(\Omega)] = L_1 [i_A] + L_2 [i_B]$ 
measures the distortion
of lengths that is induced by this homeomorphism. 
\end{ex}

\begin{figure}[htp]
\begin{center}
 \includegraphics[scale=0.8]{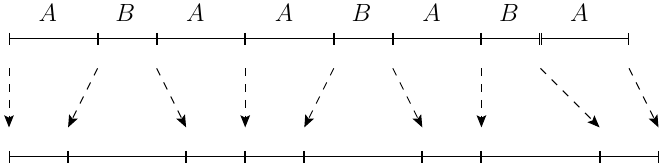}
 \caption{The two tilings have the same combinatorics. Tiles of type $a$ have length $1.5$ on top and $1$ on the bottom; tiles of type $b$ have length $1$ and $2$ respectively on top and bottom.
 The shape-changing deformation induces a piecewise linear map $\mathbb R \rightarrow \mathbb R$, with two different linear parts: either contraction by a factor~$2/3$, or expansion by a factor~$2$.}
\end{center}
\end{figure}

\begin{ex} \label{example-translation}
Let $\Omega$ be any FLC tiling space, and let $s: \Omega \to \R^d$
be a weakly PE function. Suppose further that
$\nr{s(T-x_1)-s(T-x_2)} < \nr{x_1-x_2}$ for all $T \in \Omega$ and $x_{1,2} \in \R^d$.
Define $h: \Omega \to \Omega$ by $h(T) = T - s(T)$.
We call the map map $h$ a \emph{quasi-translation}, as defined before, and the additional assumption
that $s$ is (strictly) $1$-Lipschitz ensures that it is a homeomorphism
from $\Omega$ to itself. It needs not, however, preserve transversals.
Thus $h^*(\Delta x)$ is not well-defined. However, $h^*[\ron F]$ is 
well-defined, and equals $[\Delta x]$. The reason is that $h$ is homotopic to
the identity map, and homotopic maps induce the same pullback map on
$\check H^1(\Omega, \R^d)$. Since the class of the identity homeomorphism
is $[\Delta x]$, so is the class of $h$.
We will denote such a homeomorphism $\tau_s$.
\end{ex}

These two examples may appear special, but in fact {\em any}
homeomorphism can be written as the combination of a \qtran{} and a
map that sends transversals to transversals.

\begin{prop}[\cite{RS09}]\label{prop:approx-RS}
  If $h: \Omega \ra \Omega'$ is a homeomorphism between FLC tiling
  spaces, 
  then there exists $s: \Omega \ra \R^d$, continuous and arbitrarily
  small, such that the map $T \mapsto h(T)-s(T)$ maps FLC transversals of $\Omega$ to
  FLC transversals of $\Omega'$.
\end{prop}

If we take $h_s(T) = h(T)-s(T)$, then $h$ and $h_s$ are homotopic maps, so
$[h]=h^*[\ron F]= h_s^*[\ron F] = [h_s^*(\Delta x)]$. In practice, one uses the
approximating map $h_s$ to define a pull-back on PE cochains and to 
actually compute $[h]$, but the class $[h]$ is independent of the approximation.

\begin{ex} \label{example-Fibonacci}
The following example comes from a substitution system: see~\cite{MH38}, or~\cite{Fog02} for a modern account.
Let $\Omega$ be a 1-dimensional tiling space coming from
the Fibonacci substitution $A \to AB$, $B \to A$, with tile lengths
$|A|=\phi$ and $|B|=1$, where $\phi= (1+\sqrt{5})/2$ is the golden mean. 
Let $A_n$ and $B_n$ denote the patches (called {\em $n$-th order supertiles})
obtained by applying the substitution $n$ times to $A$ and $B$. (\emph{E.g.},
$A_3 = ABAAB$.)
Let $\Omega'$ be a suspension of the Fibonacci subshift with tile lengths
$|A'|=|B'|= (2+\phi)/(1+\phi)$, and whose supertiles have lengths
$|A'_n|$ and $|B'_n|$. The spaces $\Omega$ and $\Omega'$ are 
known to be topologically conjugate \cite{RS01, CS06} via a conjugacy
$h$ that preserves the sequence of letters, insofar as the differences 
in lengths $|A_n|-|A'_n|$ and $|B_n|-|B'_n|$ go to zero as $n\to\infty$. 

Since $h$ commutes with translation, one might expect $h^*(\Delta x)$ to 
equal $\Delta x$, and hence for $[h]$ to equal $\ron F(\Omega)$. However, 
{\em this is not the case}. The pullback $h^*(\Delta x)$ is not well-defined,
since $h$ does not preserve transversals. To compute $h^*[\ron F(\Omega')]$
we must find an approximation $h_s$ and compute $[h_s^*(\Delta x)]$. 

One such approximation $h_{s_0}$ is given by the procedure of Example
\ref{example-shape-change}. 
Tiles $A$ and $B$ get mapped to tiles $A'$ and 
$B'$, and $h_s^*[\ron F] = |A'| [i_A] + |B'|[i_b]$. Closer approximations
$h_{s_n}$ 
are obtained by applying the procedure of Example \ref{example-shape-change}
to $n$-th order supertiles, and the map $h$ is the limit of $h_{s_n}$ as 
$n \to \infty$. Then
\[
\begin{split}
 [h_{s_n}^* \Delta x] & = |A_n'| [i_{A_n}] + |B_n'| [i_{B_n}] \\
      & = \begin{matrix} \begin{pmatrix} |A'| & |B'|  \end{pmatrix}  \\ \mbox{} \end{matrix}
            \begin{pmatrix} 1&1 \\ 1&0 \end{pmatrix}^n 
            \begin{pmatrix} [i_{A_n}] \\ [i_{B_n}] \end{pmatrix} \\
     & = \begin{matrix} \begin{pmatrix} |A'| & |B'|  \end{pmatrix}  \\ \mbox{} \end{matrix}
           \begin{pmatrix} [i_{A}] \\ [i_{B}] \end{pmatrix} \\
     & = |A'| [i_A] + |B'| [i_B].
\end{split}
\]
Note that the answer does not depend on $n$. All maps $h_{s_n}$ are homotopic
to each other, and to $h$ itself, so all define the same pullback in
cohomology.
\end{ex}

\begin{figure}[htp]
\begin{center}
 \includegraphics[scale=0.8]{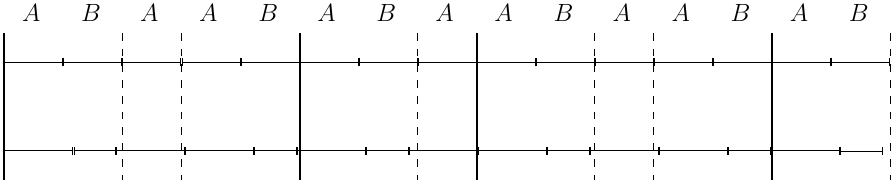}
 \caption{Even though the deformation is not trivial, the global size change of large patches stays bounded:
 the associated tiling spaces are topologically conjugate.}
\end{center}
\end{figure}

\section{A cohomological invariant for orbit equivalence}

We next study orbit equivalence between minimal, aperiodic tiling
spaces that do not necessarily have FLC.
An orbit equivalence is a homeomorphism which
sends $\R^d$-orbits to $\R^d$-orbits. In dynamics, the term ``orbit
equivalence'' is often used in the framework of $\R$-actions. In that
case, it is usually required that the homeomorphism preserve the
orientation of the orbits.  We make no such assumption here.

When a tiling space has FLC, its path components are its translational
orbits, so homeomorphisms necessarily preserve orbits.  However, this
is not true for ILC spaces. For instance, consider a 2-dimensional
tiling space $\Omega$ (such as the space of pinwheel tilings~\cite{Rad94}) 
that admits continuous rotations. Pick a continuous function $f$ on
$\Omega$ that is invariant under rotation about the origin.  Then
consider the map $T \mapsto R_{f(T)} T$, where $R_\theta$ is rotation
about the origin by an angle $\theta$.  If $f(T-x) \ne f(T)$, then it
is likely that $T-x$ and $T$ will be sent to different translational
orbits.

When two spaces are indeed orbit equivalent, there is a cocycle that
describes how an orbit is mapped to another.
\begin{defin}
 Given an orbit equivalence $h: \Omega \ra \Omega'$, the associated
 cocycle is the map
 \[
  (T,x) \mapsto h_T(x),
 \]
 defined by $h(T-x) = h(T) - h_T(x)$.
 It is well-defined, by aperiodicity.
\end{defin}

\begin{prop}\label{prop6-2}
 Given $h: \Omega\ra\Omega'$ as above, for all $T$, the map $h_T$ is
 a homeomorphism; and $(T,x) \mapsto h_T(x)$ is continuous in two
 variables.
 In addition, for all $T$, the map $\alpha(x,v) = h_T(x+v) - h_T(x)$ 
 is a weakly $T$-equivariant $1$-cocycle in the sense of Section~\ref{sec:cohomology}.
\end{prop} 

A version of this result
exists in the literature when $h: \Omega \ra \Omega$ is a map which sends each
$\R^d$-orbit to itself (see~\cite{Kwa10}, main theorem and discussion in
the last paragraph).
We include a sketch of a proof for the sake of completeness.

\begin{proof}
 Checking the algebraic cocycle condition on $\alpha$ is straightforward.
 The non-obvious part is whether $(T,x) \mapsto h_T(x)$ is continuous.

 We know that $h$ maps orbits to orbits. The problem is that the topology
 induced by $\Omega$ on orbits (simply by restriction) is not the topology
 of $\R^d$.
 However, a standard argument of foliation theory (see for example~\cite{MS06})
 guarantees
 that the orbits (or ``leaves'') can be given the topology
 of $\R^d$ without invoking the group action, but using the topology
 of $\Omega$ and its local product structure (see Prop.~\ref{prop:lamination}).
 
 Given $T \in \Omega$ and $x \in \R^d$, consider $\gamma$ a path
 from $T$ to $T-x$.
 One can choose a ``plaque path'' from $T$ to $T-x$: a covering of the path
 $\gamma$ by elements of the form $P_i := \phi_i(\{T_i\} \times B(0,r))$ in local
 charts. (The sets $P_i$ are called plaques.)
 Each $P_i$ is given the topology of $\R^n$, which also corresponds to
 the induced topology by $\Omega$.
 Then, by uniform continuity of $h$, and up to reducing the size of the
 plaques involved, each $h(P_i)$ can be included in
 an open chart domain $V_i$ of $\Omega'$.
 Therefore there is a plaque path in $\Omega'$:
 $h(T) \in P'_1, \ldots, P'_k \ni h(T-x) = h(T) - h_T(x)$
 with $h(P_i) \subset P'_i$.
 There is now a plaque path covering $h_* \gamma$.
 When $t$ tends to
 $1$, $h_* \gamma(t)$ is in $P_k$, which is homeomorphic to an open set
 of $\R^d$, so $h_T (\gamma(t))$ tends to $h_T(x)$.
 Therefore, $h_T$ is continuous at $x$.
 
 To prove joint continuity in $(T,x)$, fix $T_0, x_0$ and $\eps >0$.
 Let $\delta$ be the constant given by uniform continuity of $h$,
 corresponding with $\eps$.
 If $T,x$ are $\delta_0$-close from $T_0, x_0$ respectively (for
 an appropriate $\delta_0$), then
 $T-tx$ and $T_0 - tx_0$ are $\delta$-close from each other for all
 $t \in [0,1]$.
 It means that for all $t$, $h(T-tx)$ and $h(T_0 - tx_0)$ are
 $\eps$-close.
 Now, assume that $h_T(x)$ and $h_{T_0}(x_0)$ are not $\eps$-close.
 Because $h(T)-h_T(x)$ and a patch of $h(T_0)-h_{T_0}(x_0)$
 are $\eps$-close, the underlying Delone sets need
 to match up to a displacement of $\eps$.
 By Delone property, it means that if $h_T(x)$ and $h_{T_0}(x_0)$
 are not within $\eps$, they must be at least at distance $r$ from
 each other ($r$ is the uniform separation constant of the
 Delone set).
 By continuity of $t\mapsto h_{T_i}(tx)$, it means that there is
 a $t_0$ such that $\nr{h_T(tx) - h_{T_0}(t x_0)}$ is between,
 for example, $2\eps$ and $r/2$ (provided $\eps$ is small enough).
 This is a contradiction.
\end{proof}

\begin{defin}
\label{def:ILC-class}
Given $h: \Omega \ra \Omega'$ an orbit equivalence between aperiodic tiling spaces with or without FLC, we define a map $\alpha: \Omega \times \R^d \ra \R^d$ by $\alpha(T,v) = h_T(v)$. For any given $T$, it also defines a map (still called $\alpha$) from $\R^d \times \R^d$ to $\R^d$ by $\alpha(x,v) = h_{T-x}(v)$, which for any fixed $v$ is weakly PE in the variable $x$, with respect to the tiling $T$.

Then the class of $h$ in the weak group $H^1_w(\Omega;\R^d)$ is $[h] := [\alpha]$.
\end{defin}

This definition is consistent with the one in the FLC case. First,
note that if we write $\ron F (x,v) = v$ (the fundamental shape class
in the dynamical picture for cohomology), then a simple computation
shows that $h^*([\ron F]) = [\alpha]$ in weak cohomology.  If the tiling
space has FLC, then by Proposition~\ref{prop:approx-RS} the map $h$ is
homotopic to a map $\tilde h$ which sends FLC transversals to FLC
transversals.  Then $\tilde h^*$ and $h^*$ induce the same pullback
map $H^1_s (\Omega';\R^d) \ra H^1_s(\Omega;\R^d)$.  Furthermore, the
pullback of a \emph{strong} cochain by $\tilde h$ is then a strong
cochain, and the strong class of $h$ is then $h^* [\ron F] = \tilde
h^* [\ron F]$ in the \v Cech picture, or $(T,v) \mapsto \tilde h_T(v)$
in the dynamical picture.



If the spaces are FLC, and if $h_s$ preserves transversals (as
in Proposition \ref{prop:approx-RS}), then $\alpha_s$ defines a class 
in the {\em strong} cohomology. Different choices of $s$ yield homotopic maps,
and so define the same class. Thus, whenever the two spaces are FLC, we can
define a strong cohomology class by perturbing $h$ to a map $h_s$ that preserves
transversals and defining $[h]$ to be $[h_s]$. This is exactly the procedure
that we followed (albeit in the PE setting) in the last section.

The following result was proved in~\cite{Jul17} in the FLC case. We give here
a proof that is valid in both cases of finite and infinite local complexity.
\begin{prop}\label{prop:almost-lipschitz}
 Let $h: \Omega \ra \Omega'$ be an orbit equivalence between two (not necessarily FLC) tiling spaces.
 Then there exists $\lambda > 1$ and $C > 0$ such that for all $T \in \Omega$ and all $x \in \R^d$,
 \[
  \lambda^{-1} \nr{x} - C \leq \nr{h_T(x)} \leq \lambda \nr{x} + C.
 \]
\end{prop}

\begin{proof}
 Consider the cocycle $\alpha: \Omega \times \R^d \ra \R^d$ defined by $\alpha(T,v) = h_T(v)$.
 It is continuous, and so the image of the compact set $\Omega \times \bar B (0,1)$ is compact, hence bounded
 by, say, $\lambda$.
 Let now $v \in \R^d$. We write trivially $v = (n+1) v/(n+1)$, where $n$ is the integer part of $\nr{v}$.
 Then for all $T$ (using the cocycle property),
 \[
 \begin{split}
  \nr{\alpha(T,v)} & = \nr{\sum_{i=0}^n \alpha\big( T-i \frac{v}{n+1}, \frac{v}{n+1} \big) } \\
     & \leq \sum_{i=0}^n \nr{\alpha\big( T-i \frac{v}{n+1}, \frac{v}{n+1} \big)} \\
     & \leq (n+1) \lambda = \lambda \nr{v} + \lambda.
 \end{split}
 \]
 The upper bound of the theorem holds true. For the lower bound, repeat the argument,
 (replacing $h$ with $h^{-1}$ and noticing that $(h_T)^{-1} = (h^{-1})_{h(T)}$).
\end{proof}

We now state and prove the main theorem of this section.
\begin{thm}\label{Classification-Thm}
 Let $h_i: \Omega \ra \Omega_i$ be two orbit-equivalences ($i \in \{1,2\}$).
 \begin{enumerate}
 \item If $[h_1] = [h_2]$ as elements of the weak cohomology, then
   there exists a continuous $s: \Omega \ra \R^d$ such that $\tau_s: T
   \mapsto T-s(T)$ is a homeomorphism, and there exists a topological
   conjugacy $\phi: \Omega_1 \ra \Omega_2$ such that $h_2 \circ \tau_s
   = \phi \circ h_1$.
  \[
   \begin{CD}
    \Omega   @>{h_1}>> \Omega_1  \\
     @V{\tau_s}VV           @VV{\phi}V \\
    \Omega   @>{h_2}>> \Omega_2
   \end{CD}
   \]
 \item If the tiling spaces are FLC and $[h_1] = [h_2]$ in the strong
cohomology, then the same statement holds, only with $\phi$ an MLD map.
 \end{enumerate}
\end{thm}

\begin{proof}
 We begin with the first statement.  A simple computation shows that
 if $h$ is an orbit equivalence between tiling spaces, then
 $(h^{-1})_{h(T)} = (h_T)^{-1}$. 
 Let $h_1$ and $h_2$ be two
 homeomorphisms with $[h_1]=[h_2]$ in the weak group.
 By definition of the weak class associated with a homeomorphism (Definition~\ref{def:ILC-class}), 
 the cocycles associated with $h_1$ and $h_2$ differ by a co-boundary, \ie
 for all~$T$ and~$x$,
 \begin{equation}\label{eq:weak-coboundary}
  h_{2,T}(x) = h_{1,T}(x) + \bigl( s_0 (T - x) - s_0(T) \bigr).
 \end{equation}
 Now, for $T' \in \Omega_1$ and $w \in \R^d$, let us compute $h_2 \circ h_1^{-1}(T'-w)$.
 Let $T = h_1^{-1}(T')$.
 \begin{align*}
  h_2 \circ h_1^{-1}(T' \!-\! w) & = h_2 \big( T \!-\! h_{1,T}^{-1} (w) \big)  \\
    & = h_2(T) \!-\! h_{2,T} \circ h_{1,T}^{-1} (w)  \\
    & = h_2(T) \!-\! h_{1,T} \circ h_{1,T}^{-1} (w) \!-\! \Big( s_0 \big( T\!-\!h_{1,T}^{-1} (w) \big) \!-\! s_0 (T) \Big)  \\
    & = h_2 \circ h_1^{-1} (T') \!-\! w \!-\! s_0 \big( h_1^{-1}(T') \!-\! (h_{1}^{-1})_{T'}(w) \big) \!+\! s_0 \big( h_1^{-1}(T') \big)  \\
    & = h_2 \circ h_1^{-1} (T') \!-\! w \!-\! s_0 \circ h_1^{-1} (T' \!-\! w) \!+\! s_0 \circ h_1^{-1} (T').
 \end{align*}
 Now, given $T' \in \Omega'$, it is possible to define
 \[
  \phi (T') = h_2 \circ h_1^{-1} (T') + s_0 \circ h_1^{-1} (T').
 \]
 The computation above shows that $\phi(T'-w) = \phi(T') - w$ and in
 particular it is a bijection on each orbit. 
 Since $h_1$ and $h_2$
 are bijective, $\phi$ induces a bijection between the set of orbits
 $\Omega / \R^d \ra \Omega' / \R^d$.  Therefore, $\phi$ is
 bijective. It is continuous and is therefore (by compactness) a
 homeomorphism.  To conclude, $\phi$ is a topological conjugacy.
 A simple computation shows that $h_2^{-1} \circ \phi \circ h_1 = \tau_s$
 with
 \[
  \tau_s (T) = T - (h_{2,T})^{-1}(-s_0(T)),
 \]
 and it is in particular a homeomorphism.
 
 We now turn to the second statement, for which we assume FLC. 
 By functoriality of \v Cech cohomology, we have
 \[
  (h_2 \circ h_1^{-1}) [\ron F (\Omega_2)]_s = [\ron F (\Omega_1)]_s,
 \]
 or in other words $[h_2 \circ h_1^{-1}]_s = [\ron F]_s$.
 Using the computation from the first part of this theorem, we can
 choose an arbitrary $T_0 \in \Omega_1$, define
 \[
  \phi (T_0-x) = h_2 \circ h_1^{-1} (T_0) - x,
 \]
 and extend it to a topological conjugacy $\Omega_1 \ra \Omega_2$.
 By construction (see above), $\phi$ and $h_2 \circ h_1^{-1}$ are
 homotopic, so $[\phi]_s = [h_2 \circ h_1^{-1}]_s$.
 We now want to prove that $\phi$ is an MLD map, or in other words
 that it sends FLC transversals to FLC transversals (Proposition~\ref{prop:local}).
 Let $\Xi$ be a FLC transversal containing $T_0$.
 
 Using Proposition~\ref{prop:approx-RS}, 
 there exists a local map  $\psi: \Omega_1 \ra \Omega_2$
 that is $\eps$-close to $h_2 \circ h_1^{-1}$, and so homotopic to $\phi$. 
Without loss of generality (up to composing with a small translation),
we can also assume that $\psi(T_0)=\phi(T_0)$. 
Since $\psi$ is homotopic to $\phi$, its class in cohomology is $[\psi]_s = [\ron F]_s$.
Besides, $\psi$ is a map which preserves FLC transversals and therefore its
associated cochains differs from $\ron F$ by a strong coboundary.
This means there exists a function $s$ on $\Omega$, which is transversally
 locally constant (in particular its restriction to $\Xi$ is
 locally constant), such that
 \[
  \psi(T_0-x) = \psi(T_0) - x + s(T_0-x) - s(T_0).
 \]
 In particular,
 \[
  \psi(T_0-x) = \phi(T_0-x) + s(T_0-x) - s(T_0).
 \]
 Let $\Xi_0$ be a clopen set of $\Xi$ on which $s$ is constant.
 It means that for any $x$ satisfying $T_0-x \in \Xi_0$,
 \(
  \phi(T_0 - x) = \psi(T_0 - x) - s_0
 \)
 for some constant $s_0 \in \R^d$. We deduce:
 \[
  \phi(\Xi_0) = \psi(\Xi_0) - s_0.
 \]
 Since $\psi$ is local, and the translate of a FLC transversal is a
 FLC transversal, $\phi$ is local. It is therefore a MLD map.
\end{proof}

\section{The Ruelle--Sullivan map}
In the next two sections, we tackle the inverse problem: ``What
cohomology classes can one get from a homeomorphism?'' We will show
that these are precisely the classes that are mapped to invertible
matrices by the Ruelle--Sullivan map.

For the moment, we assume that $\Omega$ has FLC and is uniquely ergodic
with invariant probability measure $\mu$.  In this setting, the
notations $[h]$ or $[\alpha]$ will refer to the classes in the
\emph{strong} cohomology group associated to a homeomorphism $h$ or a
cocycle $\alpha$.

The Ruelle--Sullivan map for FLC tiling spaces is most easily defined
(and was first defined) using the PE de~Rham picture for cohomology.
See \cite{KP06} for details. Each class $[\alpha] \in H^1(\Omega;
\R^d)$ can be represented by an $\R^d$-valued PE 1-form, \emph{i.e.},
the assignment of a linear transformation $\alpha(T): \R^d \to \R^d$
at each point $T \in \Omega$.  The Ruelle--Sullivan map averages this
linear transformation (\emph{i.e.}, $d \times d$ matrix) over $\Omega$
with respect to the measure $\mu$.
 \[
  C_\mu ([\alpha])  = \int_\Omega \alpha(T) d\mu(T).
 \]
 By the ergodic theorem, this is the same as averaging $\alpha$ over a
 single orbit, \emph{i.e.},  over the points of a single tiling $T$, where
as usual we are identifying $x \in T$ with $T-x \in \Omega$.  More
 precisely, we can average over cubes $C_r$ of side $r$ centered at
 the origin, and take a limit as $r \to \infty$. 
Adding the differential $d\gamma$ of a bounded function $\gamma$ can only change
 the average over $C_r$ by $O(1/r)$, and so does not change the
 limiting value (see~\cite{KP06}), since 
 \begin{displaymath}
 \begin{split}
\int_{C_r} d \gamma(e_1) dx_1\cdots dx_d = & 
\int_{C_r}\frac{\partial \gamma}{\partial x_1} dx_1 \cdots dx_d \\
  = & \int_{\left [-\frac{r}2,\frac{r}2 \right ]^{d-1}} \gamma(\frac{r}{2}, x_2, \ldots x_d)-\gamma(-\frac{r}{2},x_2,\ldots, x_d) 
d^{d-1}x
\end{split}
\end{displaymath}
scales as $r^{d-1}$, as does $\int_{C_r} d\gamma(e_j)$ for all $j=1,2,\ldots, d$.  
 Thus $C_\mu([\alpha])$ is
 only a function of the cohomology class $[\alpha]$, and not of the
 specific form $\alpha$ used to represent it.

This is how the Ruelle--Sullivan class was developed for FLC tilings in
\cite{KP06}, but the construction never actually uses the
FLC condition!  A nearly
identical construction applies for weak cohomology classes on ILC
tiling spaces. The point in both settings is that any (weakly or strongly) PE
function or 1-form must be continuous on the compact space $\Omega$,
and so must be bounded. The average of $\alpha$ over $\Omega$ is always
well-defined and by the ergodic theorem can be computed by averaging
$\alpha$ over large cubes in a specific tiling $T$, and adding $d \gamma$
to $\alpha$ can only change the average over $C_r$ by $O(1/r)$.
 
Note that the map can also be defined on dynamical cocycles.  We
mentioned (Proposition~4.9) that any de Rham PE closed $1$-form can be
mapped to a PE dynamical $1$-cocycle as follows: if $\alpha$ is a
$1$-form, then the path integral $\beta(x,v) = \int_{x}^{x+v} \alpha$
is a dynamical $1$-cocycle representing the same class.  Then it is a
routine differential calculus computation to show that $C_\mu[\alpha]
\cdot v$ can be given by the average
 \[
  \lim_{r \ra \infty} r^{-d} \int_{C_r} \beta(s,v) ds.
 \]

We write $\alpha(x,v)$ as shorthand for $\alpha(T-x,v)$.  
By the ergodic theorem for uniquely ergodic measures,
the convergence of $r^{-d} \int_{C_r} \alpha(s,v) ds$ to $C_\mu([\alpha])(v)$
is uniform. That is, for any $\eps>0$ there is a radius $r_\eps$ such
that, for any tiling $T$ and any centering point $x$ and any vector $v$,
and for any $r>r_\eps$,  
\begin{equation} \label{ergodic-estimate}
  \nr{ r^{-d} \int_{C_r} \alpha(x+s,v) ds - C_\mu([\alpha]) (v)} \leq 
\eps \max\{\nr{v}, 1\}.
\end{equation}
Note that this result is uniform in $v$ as well as in $x$ and $T$. Uniformity
for $v$ with $\| v \| \le 2$ follows from compactness of the ball of 
radius 2.
Uniformity for larger vectors follows from the cocycle condition. If 
$\| v \| > 2$, pick an integer $n$ such that $1 < \| v/n \| < 2$. 
We can then write 
$$ \alpha(x+s,v) = 
     \sum_{i=1}^n \alpha\Bigl( x+s + \frac{i-1}{n}v, \frac{v}{n} \Bigr),$$
and apply equation (\ref{ergodic-estimate}) to each term on the right hand
side.

\begin{thm}\label{RS-invertible}
 Let $h: \Omega \ra \Omega'$ be an orbit-equivalence between two minimal, aperiodic  tiling spaces
 (with or without finite local complexity), with $\Omega$ uniquely ergodic. 
Then $C_\mu [h]$ is invertible.
\end{thm}

\begin{proof}
  Let $T$ be fixed, and let $\alpha(x,v) := h_{T-x}(v)$ be the
  $T$-equivariant cocycle associated with $h$.  Remember
  (Proposition~\ref{prop:almost-lipschitz}) that $\alpha$ satisfies
 \[
  \lambda^{-1} \nr{v} - C \leq \alpha(x, v) \leq \lambda \nr{v} + C,
 \]
 for some constants $\lambda$ and $C$, uniformly in $x$.
 Let $\eps > 0$ be smaller than $\lambda^{-1}/10$.
 Then, by Equation (\ref{ergodic-estimate})  
there is $r = r_\eps$ such that for all $x$,
 \[
  \nr{ r^{-d} \int_{C_r} \alpha(x + s, v) ds - C_\mu(\alpha)(v) } \leq \eps \nr{v}.
 \]
The cocycle property applied to the parallelogram $[x_1, x_1+v, x_2+v, x_2]$ gives:
 \begin{equation}\label{eq:cocycle-parallelogram}
 \begin{split}
  \nr{\alpha(x_1,v) - \alpha(x_2,v)} & = \nr{\alpha(x_1+v, x_2-x_1) - \alpha(x_1,x_2-x_1)} \\
                                     & \leq 2 \lambda \nr{x_2-x_1} + 2C.
 \end{split}
 \end{equation}
 so that if $x_1, x_2 \in C_r$, the norm $\nr{\alpha(x_1,v) - \alpha(x_2,v)}$ is bounded above by $2 \lambda \sqrt{d} r + 2C$.

We now combine these results. For all $x_1, x_2 \in C_r$ and all
 $v$, we have
 \[
  \nr{\alpha(x_1, v) - \alpha(x_2, v)}^2 = \nr{\alpha(x_1, v)}^2 + \nr{\alpha(x_2, v)}^2 - 2 \langle \alpha(x_1, v), \alpha(x_2, v) \rangle,
 \]
 so
 \[
  2 \langle \alpha(x_1, v), \alpha(x_2, v) \rangle = \nr{\alpha(x_1, v)}^2 + \nr{\alpha(x_2, v)}^2 - \nr{\alpha(x_1, v) - \alpha(x_2, v)}^2. 
 \]
The two positive terms on the right hand side 
are each bounded below by $(\lambda^{-1} \nr{v} -
 C)^2$, and the negative 
term is bounded below by $-2\sqrt{d} \nr{x_2
   - x_1} -2C$. Therefore,
 \[
  \langle \alpha(x_1, v), \alpha(x_2, v) \rangle 
\geq (\lambda^{-1} \nr{v}-C)^2 - c'
 \]
 (where $c'$ depends on $r$). For all $v$ of norm greater than a constant $M$
(which happens to equal $\lambda(4C + 2 \sqrt{C^2+3c'})/3$) 
we can exchange the constant offsets $C$ and $c'$ for a factor of 4:
 \[
  \langle \alpha(x_1, v), \alpha(x_2, v) \rangle \geq \frac{1}{4\lambda^2} \nr{v}^2.
 \]
 In particular, this quantity can be made as large as desired by
 increasing the length of $v$.  We integrate this quantity over $x_2
 \in C_r$, divide by $r^d$, and find (for all $v$ of norm greater than
 $M$):
 \[
  \left\langle \alpha(x_1, v) , r^{-d} \int_{C_r} \alpha(s,v) ds \right\rangle \geq \frac{1}{4\lambda^2} \nr{v}^2.
 \]
 Integrating again and taking the square root, we get for all $v$ of norm larger than $M$:
 \[
  \nr{r^{-d} \int_{C_r} \alpha(s,v) ds} \geq \frac{1}{2\lambda} \nr{v}.
 \]

 We conclude this proof by contradiction: assume $C_\mu[h]$ is
 singular, so that there is $w \in \R^d \setminus \{0\}$ such that
 $C_\mu[h] (w) = 0$.  We assume without loss of generality that
 $\nr{w} \geq M$.  Then the quantity
 \[
  \nr{r^{-d} \int_{C_r} \alpha(s,w) ds}
 \]
 must be both smaller than $\eps \nr{w}$ (with $\eps$ very small
 compared to $\lambda^{-1}$), and greater than $1/(2\lambda) \nr{w}$,
 which is a contradiction.  Therefore $C_\mu[h]$ is non singular.
\end{proof}

\section{Which cohomology classes are achievable?}

In the previous section we showed that the strong
cohomology class of a homeomorphism of uniquely ergodic
FLC tiling spaces, or the weak cohomology
class of an orbit-equivalence of uniquely ergodic general tiling spaces,
had to be mapped to an invertible matrix by the Ruelle--Sullivan map.
In this section we show that 
this is the only constraint on the possible
cohomology classes of homeomorphisms/orbit equivalences, and that
all possible classes can be achieved with shape change transformations.  
Before stating the result, we define shape changes. 

Suppose that $T$ is an FLC tiling in a minimal tiling space $\Omega$, 
to which we associate a Delone set $\Lambda$ 
(say, of vertices of $T$), with the added assumption that $0 \in \Lambda$. 
Suppose that $\alpha$ is a closed $\R^d$-valued
1-form on $\R^d$, and that $\alpha$ is strongly PE with respect to $T$ with
some radius $R$. 
We decorate each point in $\Lambda$ 
by the pattern of $T$ out to some radius $R_0 > R$ around that 
point,\footnote{Strictly
speaking, we need $R_0$ to be greater than $R$ plus the greatest distance
between nearest neighbors in $\Lambda$.} 
so that no information is lost in going from $T$ to $\Lambda$. 
Let $A(x) = \int_0^x \alpha$, and suppose that $A: \R^d \to \R^d$ is a 
homeomorphism.  (Since $\alpha$ is closed, the value of $A(x)$ is 
independent of the path taken from 0 to $x$.) 
We generate a new Delone set $\Lambda'$ by moving each point $p \in \Lambda$ to a new
location $A(p)$ while preserving the labels. 
We can then view $\Lambda'$ as the vertex set of a new tiling $T'$, typically
with the same combinatorics as $T$. The only difference is that
the displacement between adjacent vertices $v_1$ and $v_2$ of $T$ has been
changed from $v_2-v_1$ to $\int_{v_1}^{v_2} \alpha$. Note that the label of 
$v_1$ determines the pattern of $T$ out to distance $R_0$ around $v_1$, and
therefore exactly determines the values of $\alpha$ along a straight path from 
$v_1$ to $v_2$, and hence determines $\int_{v_1}^{v_2} \alpha$. The local patterns of 
$T'$ are thus determined from the local patterns of $T$, and  $T'$ has FLC. 

The assumption that $A$ is a homeomorphism is significant. It is easy to
construct PE forms $\alpha$ for which $A$ is not injective, or is not 
surjective. However, if $\alpha$ is pointwise close to the identity matrix,
as in the proof of Theorem \ref{lastthm}, below,
then $A$ is guaranteed to be a homeomorphism. 

For each $x \in \R^d$, let $h_\alpha(T-x)=T'-A(x)$. Note that if $p
\in \Lambda$, then $A(p)\in \Lambda'$, and $h_\alpha(T-p)=T'-A(p)$ has
a vertex corresponding to $p$ at the origin.  The location of any
other vertex $p'$ is then given by $A(p')-A(p)=\int_p^{p'} \alpha$.
If $T-p_1$ is close to $T-p_2$ in the tiling metric, then $\alpha$
takes on the same values on a neighborhood of $p_1$ as on a
neighborhood of $p_2$, so $T'-A(p_1)$ is close to $T'-A(p_2)$.  That
is, $h_\alpha$ is a continuous map from the orbit of $T$ to the orbit
of $T'$, and so extends to a homeomorphism $\Omega \to \Omega'$, where
$\Omega'$ is the orbit-closure of $T'$. Moreover, the definition of
$h_\alpha$ does not depend on the reference tiling $T$. Any tiling in
the canonical transversal of $\Omega$ would generate the same map.

We call $h_\alpha$ the \emph{shape change transformation} associated
to $\alpha$, or more simply a \emph{shape change}. Note that the
cohomology class $[h_\alpha] \in H^1_s(\Omega, \R^d)$ of this shape
change is represented in the de~Rham version of cohomology by $\alpha$
itself.

\medskip

We can also define shape changes for tiling spaces that are minimal
but are not assumed to have FLC. In that setting we merely assume that
$\alpha$ is weakly PE, and that $A=\int \alpha$ induces a homomorphism
$\R^d \to \R^d$.  In going from $T$ to $\Lambda$, we label each point
$p \in \Lambda$ by $T-p$, i.e. by a point in a transversal of
$\Omega$.  (It is usually possible to pick a much smaller set of
labels. We over-decorate to make sure that the map $\Omega \to
\Omega'$ is a homeomorphism and not a factor map.)  As before, we take
$h_\alpha(T-x)=T'-A(x)$ and extend by continuity to a homeomorphism
$\Omega \to \Omega'$. The cohomology class of $h_\alpha$, this time in
$H^1_w(\Omega,\R^d)$, is represented by the 1-form $\alpha$, exactly
as before.

\begin{thm}\label{lastthm}
  Let $\Omega$ be a minimal, aperiodic, FLC, uniquely ergodic tiling space, 
and let
$\alpha \in H^1_s (\Omega; \R^d)$ be such that $C_\mu (\alpha)$
  is invertible. Then there exists an FLC  tiling space $\Omega'$ and a
strongly PE 1-form $\bar \alpha$ representing $\alpha$ such that the shape 
deformation induced by
  $\bar \alpha$ is a homeomorphism
 \[
  h_{\bar \alpha} : \Omega \lra \Omega',
 \]
and such that $[h_{\bar \alpha}]=\alpha$. 
 
Likewise, let $\tilde \Omega$ be a minimal, aperiodic, 
uniquely ergodic tiling space with no assumption of FLC, 
and let 
$\alpha \in H^1_w (\Omega; \R^d)$ be such that $C_\mu (\alpha)$
  is invertible. Then there exists a tiling space $\tilde \Omega'$ and a
weakly PE 1-form $\bar \alpha$ such that the shape 
deformation induced by
  $\bar \alpha$ is an orbit equivalence
 \[
  h_{\bar \alpha} : \tilde \Omega \lra \tilde \Omega',
 \]
and such that $[h_{\bar \alpha}]=\alpha$. 
\end{thm}

\begin{proof} 
  If $C_\mu(\alpha)$ is an arbitrary invertible matrix $M$, then
  $C_\mu(M^{-1}\alpha)$ is the identity, where $M^{-1}$ acts on the
  second factor of $H^1_{s/w}(\Omega, \R^d) = H^1_{s/w}(\Omega, \R)
  \otimes \R^d$.  A shape change whose cohomology class is
  $M^{-1}\alpha$, followed by a linear transformation by $M$, would
  then be a shape change whose cohomology class is $\alpha$.  It is
  therefore enough to establish the theorem for the special case when
  $C_\mu(\alpha)$ is the identity matrix.

  We work first in the category of FLC tiling spaces.  Pick a
  reference tiling $T$, and let $\alpha_0$ be a strongly PE 1-form
  representing the class $\alpha \in H^1_s(\Omega, \R^d)$.  We would
  like to do a shape change by $\alpha_0$ itself. However, we do not
  know \emph{a priori} that $A_0(x)=\int_0^x \alpha_0$ gives a homeomorphism of
  $\R^d$. Instead, we will construct a strongly PE 1-form $\bar \alpha$,
  cohomologous to $\alpha_0$,
  such that $\bar \alpha$ is pointwise close to the identity matrix.  We
  then do a shape change by $\bar \alpha$.

To construct $\bar \alpha$, we convolve $\alpha_0$ with a bump function. 
For $r > 1$, let $\rho_r$ be a continuous function which satisfies the following
properties.
\begin{itemize}
 \item $\rho_r(x)$ is constant for $x$ in the cube $C_r = [-r/2, r/2]^d$, and achieves its maximum value on $C_r$;
 \item $\rho_r(x)$ is zero outside of the cube $C_r + [-1,1]^d$;
 \item $\rho_r(x) \geq 0$ and $\int_{\R^d} \rho_r = 1$.
\end{itemize}
The area of the annulus on which $\rho_r$ is not constant grows one order
of magnitude slower than the area of $C_r$, therefore it results from
the ergodic theorem that for all $\eps > 0$, there exists $r_0$ such
that for all $r > r_0$ and all $x$,
\begin{equation}\label{eq:approx-Cmu}
 \nr{\int_{\R^d} \rho_r(s) \alpha_0(x+s)  ds - C_\mu(\alpha) } 
\leq \eps \nr{C_\mu(\alpha)},
\end{equation}
where the norm is any operator norm on the space of $d \times d$ matrices.

Let $\eps < 1/4$ and pick $r$ accordingly so that the equation above
holds.
Define
\[
\bar \alpha (x):= \int_{\R^d} \rho_r(s) \alpha_0(x+s)  ds.
\]
This is a continuous, closed $1$-form.
Since $\rho_r$ has compact support, $\bar \alpha$ is
strongly pattern-equivariant.

We must show that $\bar A := \int \bar \alpha$ induces a homeomorphism 
$\R^d \to \R^d$. 
Since $d\bar A(x)=\bar \alpha (x)$ is $\eps$-close to the identity matrix
for all $x$, it is invertible.
Therefore, the inverse function theorem states that $\bar A$ is
a local diffeomorphism.
The map $\bar A$ is also one-to-one: indeed, assume $\bar A(x) 
= \bar A(x+v)$, so that
\[
 \int_x^{x+v} \bar \alpha  = 0.
\]
However, $\nr{\bar \alpha (s)  - I_d} < 1/4$ for all
$s$. Integrating, we get
\[
 \nr{\Bigl( \int_x^{x+v} \bar \alpha \Bigr)- v}
  \leq \int_x^{x+v} \nr{\bar \alpha - I_d}
  \leq \frac{1}{4} \nr{v},
\]
which is a contradiction.
In addition, $\bar A$ is onto: indeed, the image under $\bar A$ of a 
large ball $B(0,R)$
contains $B(0, (1-\eps)R)$. As $R$ tends to infinity, this proves
surjectivity.
Therefore, $\bar A$ is a global diffeomorphism, the shape change 
$h_{\bar \alpha}$ is well-defined,
and $[h_{\alpha_0}]=[\bar \alpha]$.   
Finally, $\bar \alpha$ is a weighted average of translates of $\alpha_0$, 
each of which is strongly pattern-equivariant and cohomologous
to $\alpha_0$, so  $\bar \alpha$ is cohomologous to $\alpha_0$ and 
$[h_{\bar \alpha_0}]=[\alpha_0]=\alpha$.
This completes the proof of the first half of the theorem. 

If $\Omega$ is a minimal and uniquely ergodic space of tilings,
without any assumption about FLC, then we proceed as before.  We
represent $\alpha\in H^1_w(\Omega, \R^d)$ by a weakly PE form
$\alpha_0$, and then convolve with $\rho_r$, for sufficiently large $r$,
to get a weakly PE form $\bar \alpha$ such that $\bar \alpha$ is pointwise
$\epsilon$-close to the identity matrix. The proof that $\bar A$ gives a
homeomorphism (actually diffeomorphism) of $\R^d$ is exactly as before, so the
shape change $h_{\bar \alpha}$ is well-defined and $[h_{\bar \alpha}] =
[\bar \alpha]$. All that remains is to show that $\bar \alpha$ and $\alpha_0$
represent the same class (namely $\alpha$) in $H^1_w$.

We compute
\[
\begin{split}
 \int_0^1 \bigl[ \alpha_0 (tv)\cdot v - \bar \alpha (tv)\cdot v \bigr] dt 
     & = \int_0^1 \int_{\R^d} \bigl[ \alpha_0(tv)\cdot v - \alpha_0(tv+s) \cdot v\bigr] \rho_r(s) ds\, dt \\
     & = \int_{\R^d} \rho_r(s) \int_0^1 \bigl[ \alpha_0(tv)\cdot v - \alpha_0(tv+s) \cdot v \bigr] dt \, ds
\end{split}
\]
Since $\alpha$ is closed, its integral over the boundary of the
closed parallelogram $[0, v, v+s, s]$ is zero, therefore
\[
 \int_0^1 \bigl[ \alpha_0 (tv)\cdot v - \bar \alpha (tv)\cdot v \bigr] dt
   = \int_{\R^d} \rho_r(s) \int_0^1 \bigl[ \alpha_0(ts)\cdot s - \alpha_0(v+ts) \cdot s \bigr] dt \, ds
\]
(as a side-note, this equation and equation~\eqref{eq:cocycle-parallelogram} are similar statements, using two pictures for cohomology).
The form $\alpha_0$ is pattern-equivariant on $\R^d$, so it is
bounded (say its matrix norm is bounded by $M$), and
$\nr{\alpha(x)_0\cdot s} \leq M \nr{s}$ for all $x$.
Since the support of $\rho_r$ is bounded as well, the quantity
above is bounded independent of $v$.
Proposition~\ref{prop:bounded-coboundary} guarantees that
$\bar\alpha$ and $\alpha_0$ represent the same cohomology class.

\end{proof}

\begin{cor}\label{cor:local-homeo}
  If $h: \Omega \ra \Omega'$ is a homeomorphism between aperiodic,
  FLC, uniquely ergodic tiling spaces, there exists a
  homeomorphism $\bar h$ between these two spaces, 
local in the sense of Definition~\ref{def:local-map},
such that $[h] = [\bar h]$.
Specifically, there exists a continuous function 
$s: \Omega \ra \R^d$
such that $h(T) = \bar h(T)-s(T)$.
\end{cor}

\begin{proof}
 Let $h': \Omega \ra \Omega'$ be a map (in general not a
 homeomorphism), homotopic to $h$, 
such that $h'$ maps FLC transversals to
 FLC transversals (Proposition~\ref{prop:approx-RS}).
 Let $\alpha(T,v):= h'_T(v)$ be the corresponding
 $1$-cocycle.
 Then $[\alpha] = [h]$ in $H^1_s (\Omega; \R^d)$.
 By convolving $\alpha$ with a suitable bump function, as in the
proof of Theorem~\ref{lastthm}, we obtain
 a cocycle $\bar \alpha$ representing the same cohomology class,
 such that for all $T$,
 \(
  v \mapsto \bar \alpha (T,v)
 \)
 is a homeomorphism, and for all $v$,
 \(
  T \mapsto \bar \alpha(T,v)
 \)
 is transversally locally constant (or equivalently,
 $x \mapsto \bar \alpha(T-x,v)$ is strongly $T$-equivariant).
 Then, fix a tiling $T_0$  and define
 \[
  \bar h (T_0 - v) := h' (T_0) - \bar \alpha(T_0,v).
 \]
 This map extends to a homeomorphism $\Omega \ra \Omega'$.
 
 Note that by the cocycle condition, $\bar \alpha(T,0) = 0$ for all $T$.
 Furthermore, $\alpha$ and $\bar \alpha$ are cohomologous in
 $H^1_s$, therefore differ by a strong coboundary, that is
 $\bar \alpha(T_0,v) = h'_{T_0}(v) + (s(T_0) - s(T_0 - v))$ where
 $v \mapsto s(T_0 - v)$ is strongly PE.
 In particular, if $v$ is a return vector to a small enough transversal containing $T_0$,
 $\alpha(T_0,v) = \bar \alpha(T_0,v)$.
 Therefore, $\bar h$ agrees with $h'$ on a small transversal containing $T_0$,
 say $\Xi$.
 Therefore, it sends $\Xi$ to an FLC transversal.
 Additionally, the function $T \mapsto (\bar h)_T (x)$ is
 transversally locally constant (because $\bar \alpha$ is).
 Therefore, $\bar h$ sends \emph{any} FLC transversal to
 an FLC transversal. It is therefore a local map by 
Proposition~\ref{prop:local}.
 Finally, note that $\bar h$ and $h$ are homotopic, and 
therefore define the same cohomology class.
\end{proof}

\begin{cor}
  Within FLC, uniquely ergodic spaces, the equivalence relation ``to
  be homeomorphic'' is generated by
 \begin{itemize}
  \item MLD;
  \item Shape-changing homeomorphisms.
 \end{itemize}
\end{cor}

Note that this corollary does {\em not} say that all homeomorphisms
are a combination of MLD maps and shape changes. In fact, a general
homeomorphism may also involve \qtrans. However, \qtrans{}
map a space to itself, and do not affect the equivalence relation.

\begin{cor}
  Within the category of  uniquely ergodic tiling spaces with or without
  finite local complexity, the equivalence relation ``to
  be orbit equivalent'' is generated by
 \begin{itemize}
  \item Topological conjugacies;
  \item Continuous shape-changes.
 \end{itemize}
\end{cor}

\begin{proof}
If $h: \Omega \to \Omega_1$ is an orbit equivalence, then $C_\mu([h])$ is
invertible, so by Theorem \ref{lastthm}
there exists a shape change $h_{\bar \alpha}: \Omega \to \Omega_2$ 
such that $[h_{\bar \alpha}] = [h]$ in $H^1_w(\Omega; \R^d)$. 
But then, by Theorem \ref{Classification-Thm}, 
$\Omega_1$ and $\Omega_2$ are topologically conjugate. 
\end{proof}

\section{Examples}
\label{sec:examples}

In this section we present some examples to show the difference between the FLC
and ILC categories, and the differences between strong and weak cohomology.
These differences are already apparent in 1 dimensional tilings.

Recall that the Thue--Morse tiling space is generated by the
substitution $\sigma(a)=ab, \sigma(b)=ba$. A word obtained by applying
the substitution $\sigma$ $n$ times to a letter is called an $n$-th
order {\em supertile} and is denoted $A_n$ or $B_n$. 
The space of Fibonacci sequences is the space of 
all bi-infinite words in the letters $a$ and $b$ such that every sub-word
is found in $A_n$ or $B_n$ for $n$ sufficiently large. 
To each sequence we can associate tilings with
two types of tiles, ordered in the same way as in the sequence. 

\begin{prop}
  The weak cohomology group $H^1_w$ of the Thue--Morse tiling space is
  infinitely generated over the rationals.
\end{prop}

\begin{proof}
  We work in pattern-equivariant cohomology. Recall that a closed
  weakly PE 1-cochain is a weak coboundary if and only
  if its integral is bounded (Proposition~\ref{prop:bounded-coboundary}).

For the Thue--Morse tiling space, let $\rho_n$ be a 1-cochain that
evaluates to 1 on every tile of each supertile $A_n$,
and to 0 on every tile of each supertile $B_n$. This is
strongly PE.  Pick $x\in (\frac12,1)$, 
so that $\sum x^n$ converges but $2^n x^n$ goes to infinity.
Then $\rho_x : = \sum x^n
\rho_n$ is a well-defined weakly PE 1-cochain.  However $\rho_x$ cannot be a
weak coboundary. To see this, consider a level $n+1$ supertile of type
$a$ sitting somewhere in a tiling:
$$A_{n+1} = A_n B_n = A_{n-1} B_{n-1} B_{n-1} A_{n-1},$$
where the first $A_{n-1}$ for the left is part of an $A_n$ supertile and the second
is part if a $B_n$ supertile (see Figure~\ref{fig:TM-supertile}).
The value of $\rho_x$ on each tile of the
first $A_{n-1}$ is exactly $x^n$ greater than the value of $\rho_x$ on
the corresponding tile of the second $A_{n-1}$, since they belong to
the same $m$-supertile for all $m>n$, and to corresponding
$m$-supertiles for $m<n$.  Since an $A_n$ supertile contains 
$2^n$ tiles, the integral of $\rho_x$ over the first $A_{n-1}$ is
$2^n x^n$ greater than the integral over the second. Since
$2^n x^n$ is not bounded, and since we can do this comparison for
any value of $n$, it is not possible for $\rho_x$ to have a bounded
integral.  Hence $\rho_x$ represents a nontrivial class in the weak 
cohomology of $\Omega$.

\begin{figure}[htp]
\begin{center}
 \includegraphics[scale=1]{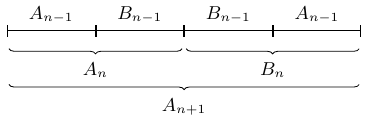}
 \caption{The inclusion of $(n-1)$-supertiles in the $n$ and $(n+1)$-supertiles.}
 \label{fig:TM-supertile}
\end{center}
\end{figure}

In any finite linear combination $\sum_k c_k \rho_{x_k}$ of such
sequences, the term with the largest $x$ will dominate on high-order
supertiles, and
$\sum_{k} c_k 2^n x_k^n $ will be unbounded as a function of $n$, 
implying that $\sum_k c_k
\rho_{x_k}$ is not a coboundary.  Thus the weak cohomology classes for
the (uncountably many!) $\rho_x$'s are linearly independent.
\end{proof}

With small modifications, the same construction could be applied to any
1-dimensional hierarchical tiling space. We note that a a very similar theorem
was proved for 1-dimensional cut-and-project sets in~\cite{BK10}.

Now let $\Omega$ be the Thue--Morse tiling space, say with both tiles of length
$1$, and let $\Omega'$ be 
a deformation of this tiling space by $\rho_x$ for some 
$x \in (\frac12,1)$.  The following result shows the difference between being
homeomorphic to an FLC tiling space and being conjugate to one.

\begin{prop}
$\Omega'$  is not conjugate to any FLC tiling space.
\end{prop}

\begin{proof}
Let $h: \Omega \to \Omega'$ be the deformation map. It is easy to see that
$[h]$ is represented by $\beta= \rho_x + \Delta x$. 
We cannot write $\beta$ 
as the sum of a weak coboundary and a strongly PE 1-cochain. To see this,
suppose that
$\beta = \beta' + d \gamma$, with $\beta'$ strongly PE with some radius $R$,
and with $\gamma$ weakly PE (hence bounded).  If $W_1$ and $W_2$ are two 
different copies of a high-level supertile within a tiling, then the 
difference between $\beta(W_1)$ and $\beta(W_2)$ would be bounded, since
the integral of $d\gamma$ is bounded and since the values of $\beta'$ at
corresponding points are
the same except on regions of length $R$ around each endpoint. However,
we have already seen that the integral of $\rho_x$ on different 
supertiles $A_n$ can differ by arbitrarily large amounts. 

Now suppose that $h': \Omega' \to \Omega''$ is a topological conjugacy
and that $\Omega''$ is an FLC tiling space. Then in weak cohomology, 
$(h' \circ h)^*[\ron F] = h^* ((h')^* [\ron F]) = h^*[\ron F]$ (since $h'$
is a conjugacy). Therefore $[h' \circ h] = [h]$ and is represented by 
$\Delta x + \rho_x$. However, $h' \circ h$ is a homeomorphism of FLC tiling
spaces, and so $[h' \circ h]$ is a class in the strong cohomology of 
$\Omega$, and in particular can be represented by a strongly PE 1-cochain. 
Since $\Delta x + \rho_x$ is not weakly cohomologous to any strongly
PE 1-cochain, we have a contradiction. 
\end{proof}

The details of the cochain $\rho_x$ are not so important to this argument. 
The important fact is that the weak cohomology is much bigger than the 
strong cohomology, so there are many orbit equivalences whose classes
are not in the image of the natural map $H^1_s \to H^1_w$. The image of
any FLC tiling space by such an orbit equivalence is necessarily
an ILC tiling space that is (by construction!) orbit equivalent 
to an FLC tiling space, but that is not topologically conjugate to any 
FLC tiling space.

\section{Continuous maps between tiling spaces}

In this final section, we explain briefly how the results of this paper
can be extended beyond orbit equivalences between tiling spaces,
to include more general surjective maps that preserve orbits. 

For any continuous map $f: \Omega \ra \Omega'$ between tiling spaces of 
finite local complexity, the induced map in \v Cech
cohomology is well defined. We define the class of the map $f$ 
in $H^1_s(\Omega; \R^d) \simeq \check H^1 (\Omega; \R^d)$ to be
\[
 [f]_s := f^* ([\ron F]).
\]
Similarly if $\Omega$ is a continuous map between arbitrary tiling spaces,
such that orbits are mapped into orbits, the equation
\[
 f(T-v) = f(T) - f_T(v)
\]
defines a dynamical $1$-cocycle $\alpha(T,v) := f_T(v)$, and the class
$[f]_w :=[\alpha]$ is well defined in $H^1_w(\Omega; \R^d)$. When $f$ is a 
homeomorphism of FLC spaces or an
orbit equivalence, these definitions agree with our previous notions. 

\begin{thm}\label{factor-thm}
Let $\Omega$ be a tiling space with or without finite local complexity.
Assume $h_i: \Omega \ra \Omega_i$ are surjective maps to other tiling
spaces which preserve orbits, and such that $h_1$ is an orbit equivalence.
If $[h_1]_w = [h_2]_w$, then there exists a factor map 
$\phi: \Omega_1 \ra \Omega_2$ such that $h_2$ is homotopic to $\phi \circ h_1$. 
If the spaces have FLC and $[h_1]_s = [h_2]_s$, then $\phi$ can be chosen to
be a local derivation. 
\end{thm}

\begin{proof}
Much of the proof of Theorem~\ref{Classification-Thm} carries over.
Since $[h_2]=[h_1]$, we must have $h_{2,T}(x)=h_{1,T}(x) + (s_0(T-x)-s_0(T))$
for some continuous function $s_0: \Omega \ra \R^d$.
We then compute
$$h_2 \circ h_1^{-1}(T'-w) = h_2 \circ h_1^{-1}(T') - w -s_0\circ h_1^{-1}(T'-w)
+ s_0\circ h_1^{-1}(T'),$$
exactly as before. 
This computation requires the invertibility of $h_1$, but makes no assumptions
on $h_2$ beyond the fact that $h_2$ preserves orbits. Also as before, we define
$$\phi(T') = h_2 \circ h_1^{-1}(T') + s_0 \circ h_1^{-1}(T'),$$
and see that $\phi(T'-w) = \phi(T')-w$. The map $\phi$ then extends to a factor 
map $\Omega_1 \ra \Omega_2$. When $\Omega_1$ and $\Omega_2$ have FLC, we
check that $\phi$ preserves a transversal (exactly as before), making $\phi$ 
a local derivation. 

Unraveling the definitions, we see that, for any tiling $T \in \Omega$, 
$\phi(h_1(T)) = h_2(T) + s_0(T),$
so
$$h_2(T) = \phi(h_1(T)) - s_0(T)$$
is homotopic to $\phi \circ h_1$. 
\end{proof}

The only important difference from the proof of
Theorem~\ref{Classification-Thm} is that we cannot write $h_2(T) +
s_0(T)$ as a composition $h_2 \circ \tau_s$.  Since $h_2$ is not
assumed to be injective on orbits, the map $(h_{2,T})^{-1}$ that we
previously used to construct the translation function $\tau_s$ is no longer
well defined.

In the uniquely ergodic case, whenever the class of the map $h_2$ has
a non-singular image under the Ruelle--Sullivan map, the existence of
the tiling space $\Omega_1$ and the orbit-equivalence $h_1$ follow from 
Theorem \ref{lastthm}:

\begin{cor}
 If $\Omega$ is uniquely ergodic and $h: \Omega \ra \Omega'$ is a 
surjective orbit-preserving map
 such that $C_\mu [h]$ is non-singular, then $h$ is homotopic to the
composition $\phi \circ h_0$ of a shape deformation $h_0$ and a factor
map $\phi$.  
 In the FLC case, the same statement holds with $\phi$ a local
 derivation.
\end{cor}

The result above shows that given an ``on average non-singular'' map
$h: \Omega \ra \Omega'$, the lack of bijectivity can have two possible
causes.  The map $h_{T}$ can fail to be bijective as a map from
$\R^d$ to $\R^d$, in which case $h$
doesn't send a given orbit to its image bijectively.
Or $h$ can collapse several
orbits into one, in which case $\phi$ is a non-invertible factor
map.

Factor maps between linearly repetitive tiling spaces have been studied in 
\cite{CDP10} (see also the review paper~\cite{ACCDP15}). The results of this 
section suggest that, by separately studying shape deformations and factor maps,
we can gain an understanding of how arbitrary tiling spaces are related. 

\begin{small}

\bibliographystyle{abbrv}

\bibliography{biblio}

\end{small}

\end{document}